\documentclass[11pt,oneside,a4paper,reqno]{amsart}

\usepackage[utf8]{inputenc}
\usepackage[T1]{fontenc}
\usepackage[english]{babel}
\usepackage[dvipsnames,svgnames]{xcolor}
\usepackage{tikz}

\usepackage{lmodern}
\usepackage{newtxtext} 

\usepackage[
activate={true,nocompatibility},
final,
tracking=true,
spacing=true,
stretch=15,
shrink=15
]{microtype}
\microtypecontext{spacing=nonfrench}
\SetTracking{encoding=*,family=*}{-20}

\usepackage[
a4paper,
margin=1.1in,
top=1.2in,
bottom=1.2in,
headheight=15pt,
headsep=0.4in,
footskip=1.2cm
]{geometry}
\usepackage{enumerate}
\definecolor{linkcol}{HTML}{B02D10}
\definecolor{winfull}{HTML}{7B2D3A}
\definecolor{wintrade}{HTML}{C99AA2}
\definecolor{hlred}{HTML}{9B2C2C}
\definecolor{rowgray}{gray}{0.95}

\usepackage{amsmath,amssymb,amsthm}
\usepackage{bm}
\newcommand{\revadd}[1]{{\color{black}#1}}

\usepackage{hyperref}
\hypersetup{
	hidelinks,
	colorlinks=false,
	citecolor=linkcol,
	linkcolor=linkcol,
	urlcolor=linkcol,
	breaklinks=true
}

\numberwithin{equation}{section}

\theoremstyle{plain}
\newtheorem{theorem}{Theorem}[section]
\newtheorem{lemma}[theorem]{Lemma}

\newtheorem{corollary}[theorem]{Corollary}

\theoremstyle{definition}
\newtheorem{definition}[theorem]{Definition}
\newtheorem{example}[theorem]{Example}

\theoremstyle{remark}
\newtheorem{remark}[theorem]{Remark}

\newcommand{\R}{\mathbb{R}}

\newcommand{\kernel}{\mathcal{K}}

\newcommand{\sphere}{\mathbb{S}^{n-1}}
\DeclareMathOperator{\diam}{diam}
\DeclareMathOperator{\esssup}{ess\,sup}
\DeclareMathOperator{\essinf}{ess\,inf}
\newcommand{\abs}[1]{\left\lvert#1\right\rvert}
\newcommand{\norm}[1]{\left\lVert#1\right\rVert}
\newcommand{\dd}{\mathrm{d}}
\newcommand{\mean}[1]{\langle #1\rangle_{\Omega}}
\newcommand{\mymean}[1]{\langle #1\rangle}
\newcommand{\aev}{\textnormal{a.e.}}

\title[Fractional modular Poincar\'e--Wirtinger inequality]
{A fractional modular Poincar\'e--Wirtinger-type inequality\\
	in Sobolev spaces with variable exponents}

\author{G.~Di Fratta}
\address{{\textsc{Giovanni Di Fratta.}}  $^1$Dipartimento di Matematica e Applicazioni ``R.~Caccioppoli'',
	Universit\`a degli Studi di Napoli Federico II}
\email{giovanni.difratta@unina.it}

\author{A.~Fiorenza}
\address{{\textsc{Alberto Fiorenza.}} $^1$Dipartimento di Architettura, Universit\`a degli Studi di Napoli Federico II, Via Monteoliveto 3, 80134 Napoli, Italy.
	$^2$Istituto per le Applicazioni del Calcolo ``Mauro Picone'', Sezione di Napoli, Consiglio Nazionale delle Ricerche, Via Pietro Castellino 111, 80131 Napoli, Italy}
\email{fiorenza@unina.it}

\author{L.~Happ}
\address{{\textsc{Leon Happ.}} Departamento de Matem\'aticas, 
	Universidad Aut\'onoma de Madrid, 
	C/ Francisco Tom\'as y Valiente 7,
	28049 Madrid, Spain}
\email{leon.happ@uam.es}

\date{\today}


\begin{document}
	
	\begin{abstract}
		Let $\Omega\subset\R^n$ be a bounded Lipschitz domain and let
		$p:\Omega\to[1,\infty)$ be a measurable, essentially bounded exponent;
		no regularity of $p$ -- in particular, no log-H\"older continuity -- is
		assumed. For a general non-negative measurable kernel $\kernel(x,y)$
		admitting a radial bounding function $\kappa(x,r)$ in the second
		variable, we establish the fractional modular
		Poincar\'e--Wirtinger-type inequality
		\[
			\iint_{\Omega\times\Omega}\kernel(x,y)\,\abs{u(x)-u(y)}^{p(x)}\dd y\,\dd x
			\leqslant C\int_\Omega c(x)\left(\int_\Omega
			\frac{\abs{\nabla u(z)}^{p(x)}}{\abs{z-x}^{n-1+\theta}}\,\dd z\right)\dd x,
		\]
		valid for every $u\in W^{1,1}(\Omega)$, in particular every
		$u\in W^{1,p(\cdot)}(\Omega)$, and every
		$\theta\in[0,1)$. Here
		\[
		c(x)=\int_0^{D}r^{n+p(x)-2+\theta}\kappa(x,r)\,\dd r,
		\]
		$D$ is the diameter
		of any fixed ball containing $\overline\Omega$, and the constant $C$ depends
		only on $n,p^+,\theta$ and the fixed domain $\Omega$.
		For the fractional kernel $\kernel(x,y)=\abs{x-y}^{-n-\lambda p(x)}$ we
		obtain a variable-exponent Gagliardo-modular estimate, while the
		endpoint $\theta=0$ with $\kernel\equiv 1$ recovers -- and extends from
		smooth functions to the full Sobolev class -- the modular
		Poincar\'e--Wirtinger inequality recently obtained in
		\cite{DDFH2024}. The proof replaces the starshaped chain decompositions
		used there by a combination of a Morrey-type estimate on starshaped
		domains, a positive-kernel proof of weighted Poincar\'e inequalities,
		and one reflection extension for the weights
		$\abs{z-x}^{-(n-1+\theta)}$. The constants are uniform down to
		exponent $1$, and the inequalities are understood in the extended
		nonnegative sense when either side is infinite.
	\end{abstract}
	
	\maketitle

	\section{Introduction}\label{sec:intro}
	
	For a bounded Lipschitz domain $\Omega\subset\R^n$ and $1\leqslant p<\infty$, the classical Poincar\'e--Wirtinger inequality is given by
	\begin{equation}\label{eq:classical-PW}
		\norm{u-\mean{u}}_{L^p(\Omega)}
		\leqslant C\norm{\nabla u}_{L^p(\Omega)},\qquad
		\mean{u}=\frac1{\mathcal{L}^n(\Omega)}\int_\Omega u\,\dd z.
	\end{equation}
	Here and throughout, a domain is defined as an open and connected set. The constant $C$ depends exclusively on the domain $\Omega$ and the exponent $p$. For a comprehensive background, we refer the reader to \cite{Brezis2011,Ziemer1989}. Furthermore, for discussions regarding the specific Poincar\'e inequality applied to functions that vanish on the domain's boundary, along with other related extensions, see \cite{DiFrattaFiorenza2022, DiFrattaFiorenza2025}.
	
	A closely related estimate concerning the fractional Gagliardo seminorm is given by 
	\begin{equation}\label{eq:gagliardo-classical}
		\iint_{\Omega\times\Omega}
		\frac{\abs{u(x)-u(y)}^p}{\abs{x-y}^{n+\lambda p}}\,\dd y\,\dd x
		\leqslant C\int_\Omega\abs{\nabla u}^p\,\dd z,
		\qquad 0<\lambda<1.
	\end{equation}
	These fractional Sobolev seminorms, along with their fundamental connections to first-order energies, form the cornerstone of the Bourgain--Brezis--Mironescu theory \cite{BBM2001,DavoliDiFrattaGiorgioPinamonti2025,DavoliDiFrattaPagliari2026}. 

	{

{\color{black}
In the variable-exponent setting, the modular $\rho_{p(\cdot)}(u)=\int_\Omega\abs{u(x)}^{p(x)}\dd x$ and the associated Luxemburg norm give distinct formulations of these estimates; see Section~\ref{subsec:VLS}. A modular inequality with a constant independent of $u$ implies a corresponding norm inequality by rescaling (see e.g. \cite{Fiorenza2021}), whereas the converse generally fails. Thus modular estimates can require stronger assumptions; this distinction is discussed for the maximal operator in \cite[p.~301]{CruzUribeDiFrattaFiorenza2018}.

Norm inequalities are treated in \cite{DHHR2011,CruzUribeFiorenza2013,CiarletDinca2011,KaufmannRossiVidal2017,BahrouniRadulescu2018}. Log-H\"older continuity is a common sufficient hypothesis, but is not universally necessary. For example, \cite{CruzUribeLopezGarciaOjea2026} proves Poincar\'e and Sobolev inequalities under a boundary log-H\"older condition and a weaker interior oscillation condition that permits everywhere discontinuous exponents.

Our purpose is to prove a fractional modular estimate for arbitrary measurable, essentially bounded exponents $p\geqslant1$ and general nonnegative kernels on bounded Lipschitz domains. The absence of regularity assumptions is accompanied by a different right-hand side: the gradient at $z$ is raised to $p(x)$, rather than $p(z)$. The result therefore does not assert that the usual modular Poincar\'e inequality requires weaker hypotheses than its norm counterpart. Its Luxemburg-gauge consequence is recorded in Remark~\ref{rem:luxemburg-consequence}.
}

		The main result of this paper is the following fractional modular Poincar\'e--Wirtinger-type inequality.
	}

	\begin{theorem}\label{thm:main-intro}
		Let $\Omega\subset\R^n$ be a bounded Lipschitz domain, and consider an open ball $B$ containing $\overline\Omega$ with diameter $D=\diam(B)$. Assume $p:\Omega\to[1,\infty)$ is a measurable exponent such that 
		\begin{equation}1\leqslant p^-:=\essinf_\Omega p\leqslant p^+:=\esssup_\Omega p<\infty,
		\end{equation}
		and fix a parameter $\theta\in[0,1)$.
		
		Suppose the nonnegative measurable kernel $\kernel:\R^n\times\R^n\to[0,\infty)$ admits a measurable radial bounding function $\kappa:\Omega\times(0,\infty)\to[0,\infty)$, in the sense that
		\begin{equation}
			\kernel(x,y)\leqslant\kappa(x,\abs{x-y})
			\quad\text{for \aev{} }x\in\Omega\text{ and \aev{} }y\in\R^n.
		\end{equation}
		Then, there exists a constant $C=C(n,\Omega,p^+,\theta)>0$ such that every $u\in W^{1,1}(\Omega)$---and in particular every $u\in W^{1,p(\cdot)}(\Omega)$---satisfies the following inequality:
		\begin{equation}\label{eq:intro-main}
			\iint_{\Omega\times\Omega}\kernel(x,y)
			\abs{u(x)-u(y)}^{p(x)}\,\dd y\,\dd x
			\leqslant C\int_\Omega c(x)\left(\int_\Omega
			\frac{\abs{\nabla u(z)}^{p(x)}}{\abs{z-x}^{n-1+\theta}}\,\dd z
			\right)\dd x,
		\end{equation}
		where
		\begin{equation}
			c(x)=\int_0^D r^{n+p(x)-2+\theta}\kappa(x,r)\,\dd r.
		\end{equation}
		This estimate is understood in the sense of extended nonnegative integrals, adopting the convention $0\cdot\infty=0$. Consequently, finiteness of the right-hand side guarantees finiteness of the left-hand side.
	\end{theorem}

	{
		To clarify the scope, structural properties, and immediate consequences of Theorem~\ref{thm:main-intro}, several remarks are in order.
		
		\paragraph*{\textit{Minimal regularity and finiteness}}
		The formulation in $W^{1,1}(\Omega)$ highlights the minimal regularity actually required by the proof, strengthening the result since $W^{1,1}(\Omega)$ is a larger space that contains $W^{1,p(\cdot)}(\Omega)$ on bounded domains. However, it is important to note that membership in the smaller space $W^{1,p(\cdot)}(\Omega)$ alone does not guarantee a finite fractional energy. Because the right-hand side evaluates the gradient $\nabla u(z)$ raised to a ``frozen'' exponent $p(x)$ rather than $p(z)$, this structural mismatch can cause the integral to diverge. Example~\ref{ex:infinite-modular} illustrates this distinction, and Remark~\ref{rem:finite-rhs} provides a sufficient condition for the finiteness of the right-hand side.\smallskip
		
		\paragraph*{\textit{Dependence of the constant}}
		
\revadd{The constant $C$ is} independent of the containing ball $B$, the specific kernel, and the exact point-wise values of the exponent $p$, depending on $p$ only through its essential supremum. 
\revadd{Its dependence on the fixed domain $\Omega$ is essential.} This dependence is governed by the size of $\Omega$, a fixed finite Lipschitz atlas characterizing its boundary, and its classical unweighted $L^1$ Poincar\'e constant. Consequently, the theorem does not assert a uniform bound over families of domains with uncontrolled or degenerating global geometry.\smallskip
		
		\paragraph*{\textit{Translation invariance}}
		Both sides of \eqref{eq:intro-main} are invariant under the addition of constants. Consequently, if the inequality holds for a function $u$, it equally holds for $u+c$ for any constant $c \in \R$. In particular, this allows one to assume, without loss of generality, that $u$ is a null-average (zero-mean) function whenever convenient.\smallskip
		
		\paragraph*{\textit{Fractional Gagliardo bounds}}
		When applying the main theorem to the fractional kernel $\kernel(x,y)=\abs{x-y}^{-n-\lambda p(x)}$ with $0<\lambda<1$, we require the parameters to satisfy the sufficient condition
		\begin{equation}\label{eq:intro-compat}
			\theta\in[0,1),\qquad \theta>1-(1-\lambda)p^-.
		\end{equation}
		Under this assumption, we obtain the variable-exponent Gagliardo-modular bound:
		\begin{equation}\label{eq:intro-gagliardo}
			\iint_{\Omega\times\Omega}
			\frac{\abs{u(x)-u(y)}^{p(x)}}{\abs{x-y}^{n+\lambda p(x)}}\,\dd y\,\dd x
			\overset{\text{Corollary~\ref{cor:variable-gagliardo}}}{\leqslant} C'\int_\Omega\int_\Omega
			\frac{\abs{\nabla u(z)}^{p(x)}}{\abs{z-x}^{n-1+\theta}}\,\dd z\,\dd x.
		\end{equation}
		The restriction in \eqref{eq:intro-compat} naturally arises because the coefficient $c(x)$ relies on the radial integral $\int_0^D r^{(1-\lambda)p(x)+\theta-2}\dd r$. For this integral to converge at the origin, the exponent of $r$ must be strictly greater than $-1$, a requirement globally ensured by the infimum $p^-$. Crucially, the admissible range $[0,1)\cap(1-(1-\lambda)p^-,\infty)$ is nonempty for every $\lambda\in(0,1)$; moreover, if $(1-\lambda)p^->1$, one can simply choose $\theta=0$. However, we make no claim that \eqref{eq:intro-compat} is strictly necessary if alternative methods are employed, nor do we assert that the constants remain bounded in the limit as $\lambda\uparrow1$.\smallskip
		
		\paragraph*{\textit{Recovery of the modular Poincar\'e--Wirtinger inequality}}
		Alternatively, selecting the trivial kernel $\kernel\equiv1$ alongside $\theta=0$ allows us to apply Jensen's inequality to deduce:
		\begin{equation}\label{eq:intro-modularPW}
			\int_\Omega\abs{u(x)-\mean{u}}^{p(x)}\,\dd x
			\overset{\text{Corollary~\ref{cor:modularPW}}}{\leqslant} C''\int_\Omega\int_\Omega
			\frac{\abs{\nabla u(z)}^{p(x)}}{\abs{z-x}^{n-1}}\,\dd z\,\dd x.
		\end{equation}
		
\revadd{This specialization recovers} the modular Poincar\'e--Wirtinger inequality recently established in \cite{DDFH2024}, and extends its validity from smooth functions to the entire variable-exponent Sobolev space $W^{1,p(\cdot)}(\Omega)$.\smallskip
		
		\paragraph*{\textit{Unbounded convex domains}}
		Finally, if the domain $\Omega$ is convex, the basic estimate naturally extends to unbounded sets. For instance, using the regularized kernel $(\mu+\abs{x-y}^2)^{-q}$ with $\mu>0$, the radial coefficient $c(x)$ remains uniformly bounded as long as the integrability condition $2q>n+p^+-1+\theta$ is met (see Corollary~\ref{cor:regularised}).
	}

	\subsection{Earlier work}\label{subsec:sota}
	
	The foundational theory of variable-exponent Lebesgue and Sobolev spaces (particularly regarding their Luxemburg norms and classical norm-based Poincar\'e inequalities) is comprehensively covered in \cite{DHHR2011,CruzUribeFiorenza2013,CiarletDinca2011}. However, attempts to directly translate these norm bounds into a naive modular analogue fail in general. Specifically, the inequality
	\begin{equation}\label{eq:naive-modular}
		\int_\Omega\abs{u-\mean{u}}^{p(x)}\,\dd x
		\leqslant C\int_\Omega\abs{\nabla u}^{p(x)}\,\dd x
	\end{equation}
	cannot hold universally for an arbitrary measurable exponent. This structural obstruction was already observed in the context of the Dirichlet problem by Fan, Zhang, and Zhao \cite{FanZhangZhao2005} (related Dirichlet modular inequalities for $\abs{u}\leqslant1$, under additional structural assumptions on $p$, are proved in \cite{DiFrattaFiorenza2022}) and a specific counterexample for the mean-subtracted inequality was provided in \cite[Example~2.1]{DDFH2024}. The core issue is that the modular of the function and the modular of its gradient can scale at fundamentally different rates, as demonstrated by the following elementary example.
	
	\begin{example}\label{ex:naive-fails}
		Consider the interval $\Omega=(0,2)$ and fix constants $1\leqslant p_1<p_2$. Define a step-function exponent such that $p(x)=p_1$ on $(0,1)$ and $p(x)=p_2$ on $(1,2)$. For the function $u(x)=\min\{x,1\}$, the mean over the interval is $\mean{u}=3/4$. Introducing a scaling factor $\lambda>0$, we compute the respective modulars:
		\begin{equation}\label{eq:scaling-counterexample}
			\int_0^2\abs{(\lambda u)'(x)}^{p(x)}\,\dd x=\lambda^{p_1},\qquad
			\int_0^2\abs{\lambda u(x)-\mean{\lambda u}}^{p(x)}\,\dd x \geqslant (\lambda/4)^{p_2}.
		\end{equation}
		Since $p_2 > p_1$, the ratio of the oscillation modular to the gradient modular diverges to infinity as $\lambda\to\infty$. Consequently, no uniform, $u$-independent constant $C$ can satisfy \eqref{eq:naive-modular} for this exponent.
	\end{example}

	To circumvent the failure of the naive modular inequality, various substitutes involving additional terms or modified hypotheses have been proposed in the literature, as discussed in \cite[Section~8.2]{DHHR2011} and \cite{Maeda2008}. Related modular phenomena concerning the maximal operator and Hardy--Poincar\'e inequalities have also been explored in \cite{CruzUribeDiFrattaFiorenza2018,DiFrattaFiorenza2022}. More recently, the authors of \cite{DDFH2024} established a modular estimate of the form \eqref{eq:intro-modularPW} for merely measurable exponents $p$; however, their formulation was restricted to smooth functions. Their proof relied on decomposing the domain into a finite chain of starshaped subdomains to produce explicit geometric constants, successfully recovering the classical inequality when $p$ is constant.
	
	In the fractional setting, Poincar\'e inequalities with constant exponents have been studied extensively, with notable contributions in \cite{Ponce2004,HurriSyrjanenVahakangas}. Correspondingly, variable-exponent fractional Sobolev spaces and their associated nonlocal equations are investigated in \cite{KaufmannRossiVidal2017,BahrouniRadulescu2018}. 

\revadd{The main contribution is the general-kernel fractional estimate, obtained from a single reflection extension whose weighted bounds are uniform in the pole and in the frozen exponent down to $1$. The underlying potential and weighted estimates are classical in nature; their uniformity allows their use for merely measurable $p$ without a density assumption. The resulting modular is \emph{asymmetric}.} In our formulation, the exponent $p(x)$ is ``frozen'' at the first integration variable. Crucially, the right-hand sides of our inequalities retain this frozen exponent $p(x)$ on the gradient $\nabla u(z)$. As a result, they are structurally distinct from, and do not provide bounds in terms of, the standard gradient modular $\int_\Omega\abs{\nabla u(z)}^{p(z)}\dd z$.

\revadd{Two distinct difficulties arise in passing to Lipschitz domains. The first concerns the singularity of the kernel. In Lemmas~3.4 and~3.5 of \cite{DDFH2024}, differences are split through an intermediate point and then averaged over a fixed starshaped core or overlap. With a kernel $\kernel(x,y)$, a direct repetition produces terms involving $\abs{u(x)-u(a)}^{p(x)}\int_D\kernel(x,y)\,\dd y$. For the fractional kernels considered here, this last integral is infinite whenever $D$ contains a neighbourhood of $x$, although the radial moment $c(x)$ can be finite. The splitting thus loses the vanishing difference at the diagonal that compensates for the singularity, so the previous gluing argument does not directly yield the present estimate. To retain the original endpoint difference, we instead extend the function to a containing ball and apply the segment estimate there.}

\revadd{The second difficulty is to control the extended gradient in the resulting weighted integrals. An unweighted extension estimate is insufficient, and weighted estimates available separately for each fixed exponent do not by themselves establish the required uniformity. We need a single extension operator, independent of $x$ and $q$, with bounds uniform for $x\in\Omega$ and $q\in[1,p^+]$, including $q=1$, before setting $q=p(x)$. Obtaining such bounds from general weighted extension theorems requires tracking their constants as both the exponent and the weight vary. We therefore give a direct reflection construction tailored to the translated power weights $\abs{z-x}^{-(n-1+\theta)}$, for which this dependence can be controlled explicitly; see Lemma~\ref{lem:weighted-tools} and Corollary~\ref{cor:gradient-extension}.}

	\subsection*{Organization of the paper} The paper is organized as follows. Section~\ref{sec:prelim} fixes our functional conventions and establishes the rigorous almost-everywhere segment formula. Section~\ref{sec:morrey} deduces the core Morrey-type radial estimate, while 
\revadd{Section~\ref{sec:convex} applies it to convex domains. Section~\ref{sec:weighted} establishes the uniform weighted Poincar\'e and extension estimates on Lipschitz domains.} Finally, Section~\ref{sec:main} synthesizes these tools to prove the main theorem and deduces its primary corollaries.

	\section{Preliminaries}\label{sec:prelim}
	Throughout this paper, we assume $n\geqslant1$. By a domain, we mean a nonempty, open, and connected subset of $\R^n$. We write $\mathcal{L}^k$ for $k$-dimensional Lebesgue measure and $\mathcal{H}^{n-1}$ for surface measure. In integrals, $\dd x$ and $\dd r$ abbreviate $\dd\mathcal{L}^n(x)$ and $\dd\mathcal{L}^1(r)$, respectively. The characteristic function of a measurable set $A$ is denoted by $1_A$, and matrix transposes by $A^{\mathsf{T}}$. All balls are Euclidean, $\sphere$ is the unit sphere in $\R^n$, and $\abs\sphere:=\mathcal{H}^{n-1}(\sphere)$, with the convention that $\abs{\mathbb{S}^0} = 2$ when $n=1$. Furthermore, we consider exclusively real-valued functions, and all gradients are to be understood in the weak sense. Constants denoted by $c_n$ depend only on $n$ and may change from line to line.
	
	Integrals of nonnegative measurable functions are permitted to take the value $+\infty$, and we adopt the standard measure-theoretic convention that $0\cdot\infty=0$. Because the values of integrals are insensitive both to the choice of representatives on Lebesgue null sets and to the values of kernels or weights on the diagonal, a singular kernel may be assigned the value $0$ on the diagonal without loss of generality. Finally, unless otherwise specified, any property asserted for almost every (\aev{}) point of a measurable set refers to the ambient Lebesgue measure $\mathcal{L}^n$.
	
	\subsection{Variable-exponent spaces} \label{subsec:VLS}
	Given a measurable exponent $p:\Omega\to[1,\infty)$ satisfying $p^+<\infty$, we define the modular and its associated Luxemburg norm as follows:
	\begin{equation} \label{eq:modular}
		\rho_{p(\cdot)}(f)=\int_\Omega\abs{f(x)}^{p(x)}\,\dd x,
		\qquad
		\norm f_{p(\cdot)}=
		\inf\{t>0:\rho_{p(\cdot)}(f/t)\leqslant1\}.
	\end{equation}
	The variable-exponent Lebesgue space $L^{p(\cdot)}(\Omega)$ is comprised of all measurable functions for which the Luxemburg norm---or equivalently, the modular---is finite. We then define the corresponding Sobolev space by
	\begin{equation}
		W^{1,p(\cdot)}(\Omega)
		=\{u\in W^{1,1}_{\mathrm{loc}}(\Omega):
		\rho_{p(\cdot)}(u)+\rho_{p(\cdot)}(\nabla u)<\infty\}.
	\end{equation}
	This definition is valid for both bounded and unbounded domains. In the setting of a bounded domain, the elementary pointwise inequality $\abs f\leqslant1+\abs f^{p(x)}$ ensures the continuous embedding $W^{1,p(\cdot)}(\Omega)\subset W^{1,1}(\Omega)$. Without loss of generality, we may redefine $p$ on a set of measure zero to ensure that $p^-\leqslant p(x)\leqslant p^+$ holds pointwise everywhere. Furthermore, we emphasize that no density of smooth functions is assumed within $W^{1,p(\cdot)}(\Omega)$.
	
The fundamental inequalities relating the modular and the Luxemburg norm are given by
\begin{equation}\label{eq:modular-norm}
	\min\{\norm f_{p(\cdot)}^{p^-},\norm f_{p(\cdot)}^{p^+}\}
	\leqslant\rho_{p(\cdot)}(f)
	\leqslant\max\{\norm f_{p(\cdot)}^{p^-},\norm f_{p(\cdot)}^{p^+}\},
\end{equation}
as established in \cite{DHHR2011,CruzUribeFiorenza2013}. It is relatively straightforward to show that if a modular inequality of the naive form \eqref{eq:naive-modular} were valid for all $u$, it would imply the corresponding norm inequality. Roughly speaking, this follows by applying the modular estimate to the normalized function $u/\norm{\nabla u}_{p(\cdot)}$ and subsequently invoking \eqref{eq:modular-norm} to deduce the norm constant $\max\{C^{1/p^-},C^{1/p^+}\}$; the trivial case where the gradient vanishes is naturally handled by the connectedness of the domain. Notably, the converse implication does not automatically hold.

{	
	
\revadd{Under our standing assumption $p^+<\infty$, membership in $L^{p(\cdot)}(\Omega)$ implies finiteness of the modular. In particular, for $u\in W^{1,p(\cdot)}(\Omega)$ on a bounded domain,
\[
\rho_{p(\cdot)}(u-\mean{u})
\leqslant 2^{p^+-1}\left[\rho_{p(\cdot)}(u)
+\mathcal L^n(\Omega)\max\{\abs{\mean{u}}^{p^-},\abs{\mean{u}}^{p^+}\}\right]<\infty.
\]
This uses the boundedness of the exponent as well as the integrability of constants. It does not control the double modular, where $u(y)$ is raised to $p(x)$, as the following example shows.}
	
\begin{example}[A possibly infinite double modular]\label{ex:infinite-modular}
	The purpose of this example is to demonstrate that a function can perfectly belong to the variable-exponent Sobolev space $W^{1,p(\cdot)}(\Omega)$, yet its double modular evaluates to infinity. This occurs because the ``frozen'' exponent $p(x)$ applied to the integration variable $u(y)$ can cause a local loss of integrability. Consequently, the main inequalities in this paper must allow for integrals evaluating to $+\infty$.
	
	Let $\Omega=B_2(0)\subset\R^3$, set $p=1$ on the inner ball $B_1(0)$, and set $p=3$ on the annulus $B_2(0)\setminus B_1(0)$. Next, choose a smooth cutoff function $\eta\in C_c^\infty(B_1(0))$ that is identically $1$ in a neighborhood of the origin. The function $u(z)=\eta(z)/\abs z$ is a valid member of $W^{1,p(\cdot)}(\Omega)$: near the origin, its magnitude and the magnitude of its gradient are proportional to $\abs z^{-1}$ and $\abs z^{-2}$, respectively, both of which are locally integrable in dimension $3$ under the local exponent $p=1$. Nevertheless, if we fix any point $x\in B_2(0)\setminus B_1(0)$ where the exponent is $p(x)=3$, the integral over $y$ diverges:
	\begin{equation}\label{eq:infinite-double-modular}
		\int_\Omega\abs{u(x)-u(y)}^{p(x)}\,\dd y
		=\int_\Omega\abs{u(y)}^3\,\dd y=\infty.
	\end{equation}
	This confirms that the extended-integral convention is strictly necessary even for the trivial, unweighted kernel $\kernel\equiv1$. 
	
	Furthermore, the same construction demonstrates that the fractional energy associated with the kernel $\kernel(x,y)=\abs{x-y}^{-3-3\lambda}$ is also infinite for any fractional parameter $\lambda\in(0,1)$. To see this, fix an arbitrary $x\in B_2(0)\setminus\overline{B_1(0)}$ and choose a radius $\varepsilon\in(0,1)$ sufficiently small so that $\eta=1$ identically on $B_\varepsilon(0)$. Because $x$ is strictly separated from $B_\varepsilon(0)$, the distance $\abs{x-y}$ is uniformly bounded from above, meaning the singular fractional weight $\abs{x-y}^{-3-3\lambda}$ is bounded from below by a strictly positive constant $c>0$ for all $y\in B_\varepsilon(0)$. Recalling that $u(x)=0$ and $u(y)=\abs{y}^{-1}$ on this ball, we can bound the fractional integral from below by integrating strictly over $B_\varepsilon(0)$:
	\begin{equation}
		\int_\Omega \frac{\abs{u(x)-u(y)}^3}{\abs{x-y}^{3+3\lambda}}\,\dd y
		\overset{\eta=1\text{ on }B_\varepsilon(0)}{\geqslant} c \int_{B_\varepsilon(0)}\abs y^{-3}\,\dd y
		\overset{}{=}4\pi c \int_0^\varepsilon\frac{\dd r}{r}=\infty.
	\end{equation}
	This demonstrates that membership in the variable-exponent Sobolev space $W^{1,p(\cdot)}(\Omega)$ alone is insufficient to guarantee the finiteness of the fractional energy. Consequently, it is strictly necessary to formulate the main estimates of this paper as inequalities in the extended non-negative sense, allowing for both sides to evaluate to $+\infty$.
\end{example}
}

\subsection{Starshapedness and radial bounds}

To establish our foundational Morrey-type estimates, we first consider starshaped domains. This geometry ensures that connecting segments remain strictly internal, enabling the application of the fundamental theorem of calculus along rays. We also introduce radial bounding functions to uniformly control general nonnegative kernels.

\begin{definition}\label{def:starshaped}
	An open set $U$ is starshaped with respect to a subset $S\subset U$ if
	the closed line segment $[x,y]$ connecting $x$ and $y$ satisfies $[x,y]\subset U$ for every $x\in S$ and $y\in U$.
	Equivalently, this geometric property can be expressed via the indicator function $1_U$ as
	\begin{equation}\label{eq:starshaped-chi}
		1_U(x+r\sigma)\leqslant1_U(x+t\sigma)
		\quad(x\in S,\ \sigma\in\sphere,\ 0\leqslant t\leqslant r).
	\end{equation}
\end{definition}

Note that a convex open set is trivially starshaped with respect to itself. 

In conjunction with this geometry, for any given nonnegative measurable kernel $\kernel$, we define a \emph{radial bounding function} on $S$ as a measurable mapping $\kappa:S\times(0,\infty)\to[0,\infty)$ that satisfies
\begin{equation}\label{eq:radial-bound}
	\kernel(x,y)\leqslant\kappa(x,\abs{x-y})
	\quad\text{for \aev{} }x\in S\text{ and \aev{} }y\in\R^n.
\end{equation}
We emphasize that no monotonicity of $\kappa$ with respect to its radial variable is assumed in this definition.

\subsection{Muckenhoupt weights and translated power weights}

To extend our localized estimates from starshaped sets to general bounded Lipschitz domains via reflection operators, we require specialized weighted estimates. These estimates rely crucially on the properties of Muckenhoupt $A_1$ weights. In this subsection, we briefly recall the definition of the $A_1$ class and demonstrate that the specific translated power weights emerging from our geometric distance kernels satisfy this condition uniformly.

{ 
Recall that a weight is defined as a locally integrable function on $\R^n$ that is positive almost everywhere. A weight $w$ belongs to the Muckenhoupt class $A_1(\R^n)$ if its $A_1$ characteristic (or Muckenhoupt constant), defined by
\begin{equation}\label{eq:A1-characteristic}
	[w]_{A_1}=\sup_Q\left(\frac1{\mathcal{L}^n(Q)}\int_Qw(z)\,\dd z\right)
	\esssup_Q w^{-1}(z),
\end{equation}
is finite, where the supremum is taken over all axis-aligned cubes $Q\subset\R^n$. By definition, we always have $[w]_{A_1} \geqslant 1$, with equality holding if and only if $w$ is equal to a positive constant almost everywhere. Furthermore, the $A_1$ characteristic is a purely structural quantity that is invariant under scalar multiplication, translation, and isotropic dilation. Specifically, for any $c>0$, $a \in \R^n$, and $\lambda > 0$, we have
\begin{equation} \label{eq:basicpropMuckenhoupt}
	[c w]_{A_1} = [w]_{A_1}, \qquad [w(\cdot - a)]_{A_1} = [w]_{A_1}, \qquad [w(\lambda \cdot)]_{A_1} = [w]_{A_1}.
\end{equation}

While the $A_1$ class is traditionally defined using axis-aligned cubes, using Euclidean balls yields an equivalent condition up to a dimensional constant. If $[w]_{A_1}\leqslant K$, enclosing an arbitrary ball $B_r(y)$ within a circumscribed cube $Q$ of volume $\mathcal{L}^n(Q) = 2^n r^n$ yields $\int_{B_r(y)}w \leqslant \int_Q w \leqslant K\mathcal{L}^n(Q)\essinf_Q w$. Because $\essinf_Q w \leqslant w(y)$ for almost every $y\in Q$, this directly implies the crucial pointwise bound
\begin{equation}\label{eq:A1-pointwise}
	\int_{B_r(y)}w(z)\,\dd z \overset{\eqref{eq:A1-characteristic}}{\leqslant} c_n K r^n w(y),
\end{equation}
where $c_n=2^n$. We rely heavily on this specific estimate because it provides the fundamental mechanism to bound integral potential operators uniformly via weighted pointwise values in our subsequent extension arguments. 

Strictly speaking, the exceptional null set where the pointwise bound \eqref{eq:A1-pointwise} fails initially depends on the choice of the bounding cube $Q$. However, a standard approximation argument utilizing a countable collection of rational cubes guarantees that \eqref{eq:A1-pointwise} holds simultaneously for all radii $r>0$ outside a single, common null set (see \cite[Section~7.1]{GrafakosClassical2014}).
}

\revadd{The power-weight characterization in \cite[Example~4.20, p.~145]{CruzUribeFiorenza2013}, together with translation invariance, gives the following lemma. We include the elementary estimate to exhibit its uniformity in the pole.}

\begin{lemma}\label{lem:power-A1}
	If $0\leqslant\alpha<n$, then the weight $w_a(z)=\abs{z-a}^{-\alpha}$ belongs to $A_1(\R^n)$, with a common bound $K(n,\alpha)$ that is strictly independent of the translation pole $a\in\R^n$.
\end{lemma}

{
\begin{proof}
	By translation invariance \eqref{eq:basicpropMuckenhoupt}, we may assume $a=0$. Consider a ball $B=B_r(b)$. We distinguish two cases.\smallskip
	
\noindent\emph{Case 1}. If $\abs b\geqslant2r$, the triangle inequality gives $\abs b - r < \abs z < \abs b + r$ for all $z \in B$. Since the ratio $\frac{\abs b+r}{\abs b-r}$ is decreasing in $\abs b$, it is maximized at the boundary case $\abs b = 2r$, ensuring that the ratio between the supremum and infimum of $\abs z$ on $B$ is at most $3$. Bounding the integrand by its maximum directly yields
	\begin{equation}\label{eq:MHtemp1}
		\left(\frac{1}{\mathcal{L}^n(B)}\int_B\abs z^{-\alpha}\,\dd z\right)\esssup_B\abs z^\alpha
		\leqslant \left( \frac{\esssup_B\abs z}{\essinf_B\abs z} \right)^\alpha \overset{\abs b\geqslant2r}{\leqslant} 3^\alpha.
	\end{equation}

\noindent\emph{Case 2}. If $\abs b<2r$, the triangle inequality implies $B \subset B_{3r}(0)$. Enlarging the domain of integration to this centered ball and evaluating via polar coordinates gives
	\begin{equation}\label{eq:power-ball-integral}
		\int_B\abs z^{-\alpha}\,\dd z
		\overset{B\subset B_{3r}(0)}{\leqslant} \int_{B_{3r}(0)} \abs z^{-\alpha}\,\dd z
		\overset{}{=}\frac{\abs\sphere}{n-\alpha}(3r)^{n-\alpha}.
	\end{equation}
	Because $\esssup_B\abs z^\alpha\leqslant(3r)^\alpha$ and the exact volume is $\mathcal{L}^n(B) = \frac{\abs\sphere}{n} r^n$, combining these estimates yields the explicit uniform bound for the case $\abs b<2r$
	\begin{equation} \label{eq:MHtemp2}
		\left(\frac{1}{\mathcal{L}^n(B)}\int_B\abs z^{-\alpha}\,\dd z\right)\esssup_B\abs z^\alpha
		\overset{\eqref{eq:power-ball-integral}}{\leqslant} \frac{n}{\abs\sphere r^n} \left( \frac{\abs\sphere}{n-\alpha}(3r)^{n-\alpha} \right) (3r)^\alpha = \frac{n 3^n}{n-\alpha}.
	\end{equation}
	Combining \eqref{eq:MHtemp1} and \eqref{eq:MHtemp2} yields the uniform bound $K(n,\alpha) = c_n\max\left\{3^\alpha, \frac{n 3^n}{n-\alpha}\right\}$, which depends exclusively on the dimension $n$ and the exponent $\alpha$. Here $c_n\geqslant1$ accounts for passing from Euclidean balls to axis-aligned cubes: enclose each cube in its circumscribed ball and compare their volumes. This completes the proof.
\end{proof}
}

\subsection{The segment formula and its exceptional set}

Estimates in terms of $\nabla u$ rely on the fundamental theorem of calculus
along segments. For our subsequent arguments, the exceptional set of base
points must be chosen before any exponent or kernel is introduced. The
following consequence of the absolute continuity on lines (ACL)
characterization of Sobolev functions
\cite[Theorem~2.1.4]{Ziemer1989} makes the required quantifiers explicit.
Throughout this subsection, we fix finite Borel representatives of functions
and their weak gradients.

\begin{lemma}[Radial Absolute Continuity]\label{lem:ACL-radial}
	Let $U\subset\R^n$ be open and let $u\in W^{1,1}_{\mathrm{loc}}(U)$.
	There exists a measurable set $U^*\subset U$, with
	$\mathcal{L}^n(U\setminus U^*)=0$, such that, for every $x\in U^*$,
	\begin{equation}\label{eq:segment-formula}
		u(y)-u(x)
		=\int_0^1\nabla u(x+t(y-x))\cdot(y-x)\,\dd t
	\end{equation}
	for almost every $y\in U$ satisfying $[x,y]\subset U$.
	For these endpoints, the integral is absolutely convergent.
	The set $U^*$ depends only on $U$ and the fixed representatives of
	$u$ and $\nabla u$; in particular, it is independent of any exponent
	or kernel used subsequently.
\end{lemma}

{
	\begin{proof}
		Fix $\sigma\in\sphere$. The ACL characterization
		\cite[Theorem~2.1.4]{Ziemer1989}, applied after a rotation and on a
		countable cover of $U$ by relatively compact boxes, shows that on
		almost every line parallel to $\sigma$, the restriction of $u$
		agrees \aev{} with a locally absolutely continuous function on each
		component of the intersection with $U$. Its derivative is
		$\nabla u\cdot\sigma$ \aev{} The one-dimensional fundamental theorem
		\cite[Theorem~8.2]{Brezis2011} applies to this representative.
		Passing to our fixed representative excludes only a null set of
		endpoint pairs on each such line. Fubini's theorem, in the line
		coordinate and the separation $r$, therefore gives
		\begin{equation}\label{eq:radial-ACL-intermediate}
			u(x+r\sigma)-u(x)
			\overset{\text{\cite[Thm.~8.2]{Brezis2011}}}{=}\int_0^r\nabla u(x+s\sigma)\cdot\sigma\,\dd s
		\end{equation}
		with absolute convergence for almost every $(x,r)\in U\times(0,\infty)$
		satisfying $[x,x+r\sigma]\subset U$.
		
		The exceptional set obtained so far may depend on $\sigma$.
		To obtain a single exceptional set of base points, we consider
		the open set of admissible triples
		\begin{equation}
			\mathcal A
			=\bigl\{(x,r,\sigma)\in
			U\times(0,\infty)\times\sphere:
			[x,x+r\sigma]\subset U\bigr\}.
		\end{equation}
		Its openness follows from the compactness of each admissible segment.
		Let $E\subset\mathcal A$ be the set where
		\eqref{eq:radial-ACL-intermediate} fails or its integral is not
		absolutely convergent. This set is Borel: indeed, the maps
		\begin{equation}
			(x,r,\sigma)\longmapsto
			r\int_0^1
			\bigl(\nabla u(x+tr\sigma)\cdot\sigma\bigr)^\pm\,\dd t,
			\qquad a^\pm=\max\{\pm a,0\},
		\end{equation}
		are Borel measurable by the measurability of parameter integrals.
		Their sum determines absolute convergence, and, wherever the sum
		is finite, their difference gives the signed integral.
		
		For every fixed $\sigma$, \eqref{eq:radial-ACL-intermediate} implies
		that the section $E_\sigma=\{(x,r):(x,r,\sigma)\in E\}$ is
		$\mathcal{L}^n\otimes\mathcal{L}^1$-null. Tonelli's theorem therefore
		gives a measurable set $U^*\subset U$ of full measure such that,
		writing $E_x=\{(r,\sigma):(x,r,\sigma)\in E\}$,
		\begin{equation}\label{eq:radial-ACL-sections}
			(\mathcal{L}^1\otimes\mathcal{H}^{n-1})(E_x)=0
			\qquad\text{for every }x\in U^*,
		\end{equation}
		where $\mathcal{H}^{n-1}$ is restricted to $\sphere$.
		Thus the base points have been selected independently of the
		direction and radius.
		
		Fix $x\in U^*$. By \eqref{eq:radial-ACL-sections},
		\begin{equation}
			\int_{\sphere}\int_0^\infty
			1_{E_x}(r,\sigma)\,r^{n-1}\,\dd r\,
			\dd\mathcal{H}^{n-1}(\sigma)
			\overset{\eqref{eq:radial-ACL-sections}}{=}0.
		\end{equation}
		The polar-coordinate formula identifies this integral with the
		Lebesgue measure of the exceptional endpoints $y=x+r\sigma$.
		Hence \eqref{eq:radial-ACL-intermediate} holds for almost every
		admissible endpoint. Substituting $s=tr$ and $r\sigma=y-x$ yields
		\eqref{eq:segment-formula}, with absolute convergence.
		
		For $n=1$, the same argument uses
		$\mathbb{S}^0=\{-1,1\}$ with counting measure.
		No exponent or kernel enters the construction of $U^*$.
	\end{proof}
}

	\section{A Morrey-type estimate on starshaped domains}\label{sec:morrey}
	
	The following lemma is the cornerstone of the paper. To avoid any
	confusion in the application to the bounding ball performed in
	Section~\ref{sec:main}, we denote by $U$ the (possibly unbounded)
	starshaped domain.
	
\begin{lemma}\label{lem:morrey}
	Let $U\subset\R^n$ be open (possibly unbounded), starshaped with
	respect to a measurable set $S\subset U$, and let
	$D=\diam(U)\in(0,\infty]$. Let $p:S\to[1,\infty)$ be measurable,
	let $\theta\in[0,1)$, and let $\kernel$ be a non-negative measurable
	kernel on $\R^n\times\R^n$ admitting a radial bounding function
	$\kappa:S\times(0,\infty)\to[0,\infty)$ as in \eqref{eq:radial-bound}.
	Let $u\in W^{1,1}_{\mathrm{loc}}(U)$. Then for almost every
	$x\in S$,
	\begin{equation}\label{eq:morrey-local}
		\int_U\kernel(x,y)\abs{u(x)-u(y)}^{p(x)}\dd y
		\leqslant\int_0^{D} r^{n+p(x)-2+\theta}\kappa(x,r)
		\left(\int_{B_r(x)\cap U}
		\frac{\abs{\nabla u(z)}^{p(x)}}{\abs{z-x}^{n-1+\theta}}\dd z\right)\!\dd r.
	\end{equation}
	In particular, with
	\begin{equation}\label{eq:c-of-x}
		c(x):=\int_0^{D} r^{n+p(x)-2+\theta}\kappa(x,r)\,\dd r,
	\end{equation}
	for \aev{}\ $x\in S$ with $c(x)<\infty$ one has
	\begin{equation}\label{eq:morrey-global}
		\int_U\kernel(x,y)\abs{u(x)-u(y)}^{p(x)}\dd y
		\leqslant c(x)\int_U
		\frac{\abs{\nabla u(z)}^{p(x)}}{\abs{z-x}^{n-1+\theta}}\,\dd z.
	\end{equation}
\end{lemma}

{
\begin{remark}[The weight parameter $\theta$]\label{rem:theta-weight}
	The tunable exponent $\theta \in [0,1)$ is an artificial degree of freedom introduced geometrically during the decoupling of the polar integrals. Because it directly governs the severity of the Riesz-type singularity in the denominator $\abs{z-x}^{n-1+\theta}$, we will consistently refer to $\theta$ as the \emph{weight parameter} in subsequent lemmas. 
	
	Analytically, $\theta$ acts as a mechanism to shift the burden of integrability between the spatial gradient integral and the radial constant $c(x)$. Increasing $\theta$ makes the spatial weight more singular at the pole $z=x$ (necessitating the strict upper bound $\theta < 1$ to preserve local integrability in $n$ dimensions), but simultaneously provides a stronger regularizing factor $r^\theta$ to the radial integral \eqref{eq:c-of-x}, thereby allowing us to accommodate a broader, more singular class of kernels $\kernel(x,y)$.
\end{remark}
}

\begin{proof}
	Discard the null exceptional sets produced by Lemma~\ref{lem:ACL-radial} and \eqref{eq:radial-bound}, and fix a base point $x\in S$. Because $U$ is starshaped with respect to $x$, the segment $[x,y]$ is contained in $U$ for every $y\in U$. Thus, \eqref{eq:segment-formula} holds for almost every $y \in U$. Since the exponent satisfies $p(x) \geqslant 1$, we apply Jensen's inequality to the convex function $s\mapsto s^{p(x)}$ to obtain
	\begin{equation}\label{eq:segment-jensen}
		\abs{u(x)-u(y)}^{p(x)} \overset{\substack{\eqref{eq:segment-formula}}}{\leqslant} \abs{y-x}^{p(x)} \int_0^1 \abs{\nabla u(x+t(y-x))}^{p(x)} \dd t.
	\end{equation}
	The change of variables $s=t\abs{y-x}$ then yields the one-dimensional bound
	\begin{equation}\label{eq:pointwise-ftc}
		\abs{u(x)-u(y)}^{p(x)} \overset{\eqref{eq:segment-jensen}}{\leqslant} \abs{y-x}^{p(x)-1}\int_0^{\abs{y-x}}\abs{\nabla u\left(x+s\frac{y-x}{\abs{y-x}}\right)}^{p(x)}\dd s.
	\end{equation}
	
	To obtain the $n$-dimensional estimate, we extend $\nabla u$ by zero outside $U$, multiply \eqref{eq:pointwise-ftc} by the kernel $\kernel(x,y)$, and integrate over all endpoints $y \in U$. Passing to polar coordinates $y=x+r\sigma$ (with $\dd y=r^{n-1}\dd r\,\dd\sigma$, where $\dd\sigma:=\dd\mathcal{H}^{n-1}(\sigma)$) and applying the radial bound \eqref{eq:radial-bound} yields the iterated integral
	\begin{align}
		\int_U\kernel(x,y)\abs{u(x)-u(y)}^{p(x)}\dd y
		\overset{{\eqref{eq:pointwise-ftc},\eqref{eq:radial-bound}}}{\leqslant} & \int_0^{D} r^{n+p(x)-2}\kappa(x,r)\; \cdot \notag \\
		&\qquad \cdot\;  \int_{\sphere}\int_0^r 1_U(x+r\sigma) \abs{\nabla u(x+s\sigma)}^{p(x)}\dd s\,\dd\sigma\,\dd r, \label{eq:morrey-polar}
	\end{align}
	where the outer integral is truncated at $D=\diam(U)$. By Tonelli's theorem, the non-negative integrands permit all rearrangements of the order of integration.
	
	To decouple the variables $r$ and $s$, we utilize the starshaped geometry \eqref{eq:starshaped-chi}, which ensures $1_U(x+r\sigma)\leqslant1_U(x+s\sigma)$ for $0\leqslant s\leqslant r$. For any $\theta \geqslant 0$, we multiply and divide the inner integrand by $s^{n-1+\theta}$. Bounding the extracted weight by $s^\theta \leqslant r^\theta$ and applying the inverse polar substitution $z=x+s\sigma$ (with $\dd z=s^{n-1}\dd s\,\dd\sigma$) to revert the inner double integral back to Cartesian coordinates gives
	\begin{align}
		\int_U\kernel(x,y)\abs{u(x)-u(y)}^{p(x)}\dd y  & \overset{{\eqref{eq:morrey-polar},\eqref{eq:starshaped-chi}}}{\leqslant} \int_0^{D} r^{n+p(x)-2+\theta}\kappa(x,r)\;\cdot \notag \\
		&\qquad\qquad \cdot\;\int_{\sphere}\int_0^r 1_U(x+s\sigma) \frac{\abs{\nabla u(x+s\sigma)}^{p(x)}}{s^{n-1+\theta}} s^{n-1}\dd s\,\dd\sigma\,\dd r \notag \\
		&\quad= \int_0^{D} r^{n+p(x)-2+\theta}\kappa(x,r) \left( \int_{B_r(x)\cap U}\frac{\abs{\nabla u(z)}^{p(x)}}{\abs{z-x}^{n-1+\theta}}\dd z \right) \dd r. \label{eq:morrey-local-proof-step}
	\end{align}
	This establishes the localized estimate \eqref{eq:morrey-local}.
	
	Finally, enlarging the domain of the inner integral from $B_r(x)\cap U$ to $U$ only increases the non-negative integral while making it strictly independent of $r$. Factoring this term out of the outer $r$-integral immediately yields the global estimate \eqref{eq:morrey-global}.
\end{proof}
	
	\begin{remark}\label{rem:no-loghoelder}
		No regularity of the exponent function $p$ is used beyond measurability.
		The variable exponent enters \eqref{eq:pointwise-ftc} pointwise in $x$;
		all subsequent manipulations are linear in the integration variables.
	\end{remark}
	
\section{Convex domains and the double-integral form}\label{sec:convex}

We now elevate the localized pointwise estimate to a global double-integral form. Because a convex domain is unconditionally starshaped with respect to every point it contains, we can integrate the local Morrey-type estimate globally across both variables. We then specialize this general bound to two principal cases: fractional singular kernels and regularized non-singular kernels.

\begin{corollary}\label{cor:convex}
	Let $\Omega\subset\R^n$ be a convex domain (possibly unbounded) with diameter $D \in (0,\infty]$. Let the kernel $\kernel(x,y)$, its radial upper bound $\kappa(x,r)$, the exponent map $p:\Omega \to [1,\infty)$, and the parameter $\theta$ be precisely as in Lemma~\ref{lem:morrey}, setting $U=S=\Omega$. Assume $c(x)<\infty$ for almost every $x\in\Omega$. 
	
	Then, for every $u\in W^{1,1}_{\mathrm{loc}}(\Omega)$---and in particular for every $u\in W^{1,p(\cdot)}(\Omega)$ when $p^+<\infty$---we have the global double-integral bound:
	\begin{equation}\label{eq:convex-double}
		\iint_{\Omega\times\Omega}
		\kernel(x,y)\abs{u(x)-u(y)}^{p(x)}\dd y\,\dd x
		\leqslant\int_\Omega c(x)\!\left(\int_\Omega
		\frac{\abs{\nabla u(z)}^{p(x)}}{\abs{z-x}^{n-1+\theta}}\,\dd z\right)\dd x.
	\end{equation}
\end{corollary}

\begin{proof}
	Because a convex set is starshaped with respect to all of its points, Lemma~\ref{lem:morrey} applies directly with $U=S=\Omega$, yielding the localized estimate \eqref{eq:morrey-global} for \aev{}\ $x\in\Omega$. Integrating this inequality with respect to $x$ over $\Omega$ immediately yields \eqref{eq:convex-double}.
\end{proof}

We single out two immediate specializations.

\begin{corollary}[Fractional kernel, bounded convex domain]\label{cor:fractional}
	Let $\Omega\subset\R^n$ be a bounded convex domain with diameter $D$. Let $\lambda\in(0,1)$ and define $\kernel(x,y)=\abs{x-y}^{-n-\lambda p(x)}$ with radial bound $\kappa(x,r)=r^{-n-\lambda p(x)}$ for $x\in\Omega$ and $r>0$. We set the kernel to zero on the diagonal and for $x\notin\Omega$. Assume $p:\Omega\to[1,\infty)$ is measurable, $p^+<\infty$, and that the weight parameter $\theta$ from Lemma \ref{lem:morrey} satisfies
	\begin{equation}\label{eq:theta-range}
		\theta\in[0,1) \quad \text{with} \quad \theta>1-(1-\lambda)p^-.
	\end{equation}
	Then there exists a constant $C=C(D,n,\lambda,\theta,p^-,p^+)$ such that the global estimate \eqref{eq:convex-double} holds for this kernel; i.e., for every $u\in W^{1,1}_{\mathrm{loc}}(\Omega)$, and in particular every $u\in W^{1,p(\cdot)}(\Omega)$,
	\begin{equation}
		\iint_{\Omega\times\Omega}
		\frac{\abs{u(x)-u(y)}^{p(x)}}{\abs{x-y}^{n+\lambda p(x)}}\dd y\,\dd x
		\leqslant C\int_\Omega\!\int_\Omega
		\frac{\abs{\nabla u(z)}^{p(x)}}{\abs{z-x}^{n-1+\theta}}\dd z\,\dd x.
	\end{equation}
\end{corollary}

\begin{proof}
	For this specific kernel, evaluating the inner radial integral yields
	\begin{equation}\label{eq:convex-fractional-capacity}
c(x)\overset{\eqref{eq:c-of-x}}{=}\int_0^{D} r^{(1-\lambda)p(x)+\theta-2}\dd r
		=\frac{D^{(1-\lambda)p(x)+\theta-1}}{(1-\lambda)p(x)+\theta-1}.
	\end{equation}
	Under condition \eqref{eq:theta-range}, the denominator is bounded uniformly below by $(1-\lambda)p^-+\theta-1>0$, guaranteeing the integral converges. Because the map $q\mapsto D^{(1-\lambda)q+\theta-1}/((1-\lambda)q+\theta-1)$ is 
\revadd{continuous} on the compact interval $[p^-,p^+]$ and bounded away from zero in the denominator, $\esssup_{\Omega} c(x) \leqslant C < \infty$. The double-integral bound then follows directly from Corollary~\ref{cor:convex}.
\end{proof}

\begin{corollary}[Regularized kernel]\label{cor:regularised}
	Let $\Omega\subset\R^n$ be a convex domain (possibly unbounded). For parameters $\mu>0$ and $q>0$, define the regularized kernel $\kernel(x,y)=(\mu+\abs{x-y}^2)^{-q}$ with radial bound $\kappa(x,r)=(\mu+r^2)^{-q}$. Assume $p:\Omega\to[1,\infty)$ is measurable with $p^+<\infty$, $\theta\in[0,1)$, and the decay condition
	\begin{equation}\label{eq:reg-condition}
		2q>n+p^+-1+\theta
	\end{equation}
	holds. Then there exists a constant $C>0$ such that $\esssup_{\Omega} c(x) \leqslant C < \infty$, and the global estimate \eqref{eq:convex-double} holds for this kernel; { i.e., for every $u\in W^{1,1}_{\mathrm{loc}}(\Omega)$, and in particular every $u\in W^{1,p(\cdot)}(\Omega)$,
	\begin{equation}
		\iint_{\Omega\times\Omega}
		\frac{\abs{u(x)-u(y)}^{p(x)}}{(\mu+\abs{x-y}^2)^q}\dd y\,\dd x
		\leqslant C \int_\Omega\!\int_\Omega
		\frac{\abs{\nabla u(z)}^{p(x)}}{\abs{z-x}^{n-1+\theta}}\dd z\,\dd x.
	\end{equation}
}
\end{corollary}

\begin{proof}
	We must verify that $c(x)$ is uniformly bounded. Near the origin $r=0$, the integrand $r^{n+p(x)-2+\theta}(\mu+r^2)^{-q}$ is bounded above by $r^{n+p^--2+\theta}\mu^{-q}$. Because $n\geqslant 1$, $p^-\geqslant 1$, and $\theta\geqslant 0$, the exponent $n+p^--2+\theta \geqslant 0 > -1$, ensuring integrability at zero without further conditions. 
	
	At infinity, the integrand decays asymptotically as $r^{n+p(x)-2+\theta-2q}$, which is bounded above by $r^{n+p^+-2+\theta-2q}$ for $r\geqslant1$. The decay condition \eqref{eq:reg-condition} guarantees this bounding exponent is strictly less than $-1$, ensuring integrability at infinity. Thus, $c(x)$ is bounded uniformly by a constant $C$ independently of $x$. Note that if $\Omega$ is bounded, the integration truncates at $D<\infty$, rendering the decay condition \eqref{eq:reg-condition} completely unnecessary.
\end{proof}

\begin{remark}[Constant-exponent simplification]\label{rem:constant-exp}
	If the exponent $p$ is strictly constant, $\Omega$ is bounded, and $c(x)\equiv c<\infty$, the integral completely decouples. Reversing the order of integration in \eqref{eq:convex-double} by Tonelli's theorem allows us to independently bound the geometric inner integral:
	\begin{equation}\label{eq:kernel-mass-theta}
		\int_\Omega\frac{\dd x}{\abs{x-z}^{n-1+\theta}}
		\overset{\Omega\subset B_D(z)}{\leqslant}\int_{B_D(z)}\frac{\dd x}{\abs{x-z}^{n-1+\theta}}
		=\abs{\sphere}\int_0^D\rho^{-\theta}\dd\rho
		=\frac{\abs{\sphere}\, D^{1-\theta}}{1-\theta},
	\end{equation}
	which converges because $\theta<1$. This yields the simplified single-integral bound
	\begin{equation}
		\iint_{\Omega\times\Omega}\kernel(x,y)\abs{u(x)-u(y)}^p\dd y\,\dd x
		\overset{\substack{\eqref{eq:convex-double},\eqref{eq:kernel-mass-theta}}}{\leqslant}\frac{c\,\abs{\sphere}\,D^{1-\theta}}{1-\theta}
		\int_\Omega\abs{\nabla u(z)}^p\dd z.
	\end{equation}
	In particular, substituting the fractional kernel from Corollary~\ref{cor:fractional} into this decoupled estimate exactly recovers the classical fractional Gagliardo bound \eqref{eq:gagliardo-classical} on convex domains.
\end{remark}
\section{\texorpdfstring{\revadd{Weighted estimates and extension on Lipschitz domains}}{Weighted estimates and extension on Lipschitz domains}}\label{sec:weighted}
\subsection{A potential estimate with the unweighted mean}

For any function $v \in L^1 (\Omega)$ and any measurable set $A \subset \Omega$
with $0<\mathcal{L}^n(A)<\infty$, we denote the unweighted mean of $v$ over $A$ by
\begin{equation}\label{eq:patch-mean}
	\mymean{v}_A = \frac{1}{\mathcal{L}^n(A)}  \int_A v (x) \hspace{0.17em} \dd x.
\end{equation}
Here and below, a domain is understood to be open and connected.
\begin{lemma}
	\label{lem:potential-representation}Let $\Omega \subset \R^n$ be a bounded
	Lipschitz domain. There exists a constant $A_{\Omega} > 0$, depending
	exclusively on $\Omega$, such that every $v \in W^{1, 1} (\Omega)$ satisfies
	\begin{equation}
		\label{eq:sobolev-rep} \abs{v (z) - \mean{v}} \leqslant
		A_{\Omega}  \int_{\Omega} \frac{\abs{\nabla v (y)}}{\abs{z - y}^{n - 1}}
		\hspace{0.17em} \dd y = A_{\Omega} \cdot I_1 [\abs{\nabla v}] (z)  \quad
		\text{for \aev{} } z \in \Omega,
	\end{equation}
{ where $I_1$ denotes the localized Riesz potential of order $1$
		on $\Omega$.}
\end{lemma}

\begin{proof}
	If $n=1$, then $\Omega$ is an interval and the assertion follows with
	$A_{\Omega}=1$ from the one-dimensional Sobolev fundamental theorem
	\cite[Theorem~8.2]{Brezis2011}. We may therefore assume $n\geqslant2$.
	{ We proceed by a finite covering via image patches of convex
		boxes. By compactness of $\partial\Omega$, finitely many boundary
		charts, together with finitely many interior boxes, give a covering
		$\Omega = \bigcup_{j = 1}^m U_j$, where each patch $U_j = F_j
		(V_j)$ is the image of a bounded convex rectangular box $V_j \subset \R^n$
		under a bi-Lipschitz map $F_j$. For convenience, we fix a common
		bi-Lipschitz constant $L_j \geqslant 1$ for $F_j$ and its inverse.
		
		For interior patches $U_j \Subset \Omega$, $F_j = \mathrm{id}$ on $V_j =
		U_j$ (and $L_j = 1$). For boundary patches, after a rigid motion, $V_j =
		Q'_j \times (0, h_j)$ and $F_j$ is the vertical shear transformation
		\begin{equation}\label{eq:potential-flattening}
			F_j (\xi', t) = (\xi', \varphi_j (\xi') + t),
		\end{equation}
		where $\varphi_j : Q'_j \to \R$ is Lipschitz. These maps satisfy
		$\abs{\det DF_j}=1$ \aev{} Hence the change-of-variables formula for
		Lipschitz maps gives $\mathcal{L}^n(U_j)=\mathcal{L}^n(V_j)$ and
		$\mymean{v\circ F_j}_{V_j}=\mymean{v}_{U_j}$; see
		\cite[Section~3.3.3]{EvansGariepy2015}.}
	
	Set $v_j = v \circ F_j \in W^{1,1}(V_j)$. The bi-Lipschitz chain rule
	\cite[Theorem~2.2.2, pp.~52--53]{Ziemer1989} gives
	$\abs{\nabla v_j(\xi)}\leqslant L_j\abs{\nabla v(F_j(\xi))}$ \aev{}
	Applying the potential estimate on convex domains
	\cite[Lemma~7.16, p.~162]{GilbargTrudinger2001} yields
	\begin{equation}\label{eq:convex-potential}
		\abs{v_j (\xi) - \mymean{v_j}_{V_j}} \overset{\text{\cite[Lem.~7.16]{GilbargTrudinger2001}}}{\leqslant} \frac{\diam
			(V_j)^n}{n \mathcal{L}^n(V_j)}  \int_{V_j} \frac{\abs{\nabla v_j (\eta)}}{\abs{\xi
				- \eta}^{n - 1}} \hspace{0.17em} \dd \eta,
	\end{equation}
	where $\mymean{v_j}_{V_j}$ is the unweighted patch mean on $V_j$.
	
	We now transport this bound back to the curved patch $U_j$. We set $z = F_j
	(\xi)$ and $y = F_j (\eta)$. Because $F_j$ preserves Lebesgue measure ($\dd
	y = \dd \eta$) and satisfies the Lipschitz bound $\abs{z - y}
	\leqslant L_j \abs{\xi - \eta}$, the kernel satisfies $\abs{\xi - \eta}^{1 - n} \leqslant L_j^{n - 1} \abs{z - y}^{1 -
		n}$. Substituting these bounds (which contribute $L_j$ from the gradient
	and $L_j^{n - 1}$ from the denominator) yields
	\begin{equation}
		\label{eq:local-patch-rep} \abs{v (z) - \mymean{v}_{U_j}} \overset{\eqref{eq:convex-potential}}{\leqslant} A_j 
		\int_{U_j} \frac{\abs{\nabla v (y)}}{\abs{z - y}^{n - 1}} \hspace{0.17em}
		\dd y \quad \text{for } \text{\aev{} } z \in U_j,
	\end{equation}
	where $A_j = L_j^n \diam (V_j)^n / (n \mathcal{L}^n(V_j))$.
	
	To replace the local patch mean $\mymean{v}_{U_j}$ with the global mean
	$\mean{v}$, we invoke the classical unweighted $L^1$
	Poincar{\'e}--Wirtinger inequality on the connected domain $\Omega$
	\cite[Theorem~5.11.1, pp.~261--262]{Ziemer1989}:
	\begin{equation}
		\label{eq:global-pw-mean} \abs{\mymean{v}_{U_j} - \mean{v}}
		\overset{\eqref{eq:patch-mean}}{\leqslant} \frac{1}{\mathcal{L}^n(U_j)}  \int_{\Omega} \abs{v (x) -
			\mean{v}} \hspace{0.17em} \dd x \overset{\eqref{eq:classical-PW}}{\leqslant}
		\frac{C_{\Omega}}{\mathcal{L}^n(U_j)}  \int_{\Omega} \abs{\nabla v (y)}
		\hspace{0.17em} \dd y.
	\end{equation}
	Setting $d = \diam (\Omega)$, we observe that $1 \leqslant d^{n - 1}
	\abs{z - y}^{1 - n}$ for all distinct $z,y\in\Omega$. Combining
	\eqref{eq:local-patch-rep} and \eqref{eq:global-pw-mean} and taking
	$A_{\Omega} = \max_{1 \leqslant j \leqslant m} A_j + C_{\Omega} d^{n - 1}
	\max_{1 \leqslant j \leqslant m} \mathcal{L}^n(U_j)^{- 1}$ yields
	\eqref{eq:sobolev-rep}.
\end{proof}

{While Lemma~\ref{lem:potential-representation} establishes a
	purely geometric pointwise bound via the Riesz potential, we must now elevate
	this to a global norm estimate in the weighted space $L^q  (\Omega, w
	\hspace{0.17em} \dd z)$. Standard weighted Poincar{\'e} inequalities typically
	control the oscillation of a function around its {\emph{weighted}} average.
	For our subsequent extension operator, however, it is geometrically critical
	to bound the deviation from the {\emph{unweighted}} Lebesgue mean
	$\mean{v}$. The following lemma achieves exactly this, using
	the $A_1$ Muckenhoupt condition to ensure the resulting constant remains
	strictly uniform across the entire exponent range $q \in [1, P]$. }

\begin{lemma}
	[Uniform weighted Poincar{\'e} inequality]\label{lem:unweighted-mean}
	{Let $\Omega \subset \R^n$ be a bounded Lipschitz domain. Let
		$P \in [1, \infty)$ be a fixed upper bound for a set of integration
		exponents, and let $K \geqslant 1$ be a fixed upper bound for the
		characteristic of a class of $A_1$ Muckenhoupt weights.}
	
	There exists a strictly uniform constant $C_P = C_P  (n, \Omega, P, K) > 0$
	such that, for every exponent $q \in [1, P]$, every $A_1$ weight $w$
	satisfying $[w]_{A_1} \leqslant K$, and every $v \in W^{1, 1} (\Omega)$,
	\begin{equation}
		\label{eq:unweighted-mean} \int_{\Omega} \abs{v (z) -
			\mean{v}}^q w (z) \hspace{0.17em} \dd z \leqslant C_P 
		\int_{\Omega} \abs{\nabla v (z)}^q w (z) \hspace{0.17em} \dd z.
	\end{equation}
	The inequality holds in $[0, \infty]$. In particular, if the right-hand side
	is finite, then $v - \mean{v} \in L^q  (\Omega, w \hspace{0.17em}
	\dd z)$ without requiring a priori weighted integrability of $v$.
\end{lemma}

\begin{proof}
	Set $d=\diam(\Omega)$. For a non-negative measurable function
	$f$ on $\Omega$, let $I_1[f]$ denote its localized Riesz potential, with
	values in $[0,\infty]$, defined by
	\begin{equation}\label{eq:localized-riesz}
		I_1 [f] (z) = \int_{\Omega} \frac{f (y)}{\abs{z - y}^{n - 1}}
		\hspace{0.17em} \dd y, \quad z \in \Omega .
	\end{equation}
	For each fixed $z \in \Omega$, the singular kernel defines a finite positive
	measure on $\Omega$. Its total mass $m (z)$ is uniformly bounded by the
	radial integral over $B_d (z)$:
	\begin{equation}\label{eq:kernel-mass}
		m (z) := \int_{\Omega} \frac{\dd y}{\abs{z - y}^{n - 1}} \overset{\Omega\subset B_d(z)}{\leqslant}
		\int_{B_d (z)} \frac{\dd y}{\abs{z - y}^{n - 1}} = \abs{\sphere} d = : a_d .
	\end{equation}
	Jensen's inequality for the probability measure
	$\dd\mu_z(y)=\dd y/(m(z)\abs{z-y}^{n-1})$ gives
	\begin{equation}
		\label{eq:kernel-jensen} (I_1 [f] (z))^q \overset{\text{Jensen}}{\leqslant} m (z)^{q - 1} I_1 [f^q]
		(z) \overset{\eqref{eq:kernel-mass}}{\leqslant} a_d^{q - 1} I_1 [f^q] (z) .
	\end{equation}
	For unbounded $f$, this follows by applying Jensen's inequality to
	$\min\{f,N\}$ and letting $N\to\infty$ by monotone convergence.
	To establish the weighted global estimate, we apply Tonelli's theorem to
	shift the Riesz potential onto the weight $w$:
	\begin{equation}
		\label{eq:tonelli-dual} \int_{\Omega} (I_1 [f] (z))^q w (z)
		\hspace{0.17em} \dd z \overset{\eqref{eq:kernel-jensen}}{\leqslant} a_d^{q - 1}  \int_{\Omega} \left(
		\int_{\Omega} \frac{f (y)^q}{\abs{z - y}^{n - 1}} \hspace{0.17em} \dd y
		\right) w (z) \hspace{0.17em} \dd z \overset{\text{Tonelli}}{=}a_d^{q - 1}  \int_{\Omega} f (y)^q
		I_1 [w] (y) \hspace{0.17em} \dd y.
	\end{equation}
	Let $M$ denote the uncentered Hardy--Littlewood maximal operator over
	axis-aligned cubes in $\R^n$. The maximal-function bound for convolution with a
	non-negative radial decreasing kernel
	\cite[Theorem~2.1.10, pp.~91--92]{GrafakosClassical2014}, applied to
	$\abs{x}^{1-n}1_{B_d(0)}(x)$, gives
	\begin{equation}\label{eq:potential-maximal}
		I_1 [w] (y) \overset{\text{\cite[Thm.~2.1.10]{GrafakosClassical2014}}}{\leqslant} c_n d \cdot M (w) (y)  \quad \text{for \aev{} } y \in
		\Omega .
	\end{equation}
	Here $c_n>0$ absorbs the dimensional comparison between balls and cubes.
	With the cube convention in \eqref{eq:A1-characteristic},
	$[w]_{A_1}=\operatorname*{ess\,sup}_{\R^n}M(w)/w$. Thus
	$M(w)\leqslant K w$ \aev{}, and therefore
	the Riesz potential of the weight is uniformly bounded by the weight itself:
	\begin{equation}
		\label{eq:riesz-maximal-bound} I_1 [w] (y) \overset{\eqref{eq:potential-maximal}}{\leqslant} c_n dM (w) (y)
		\overset{\eqref{eq:A1-characteristic}}{\leqslant} c_n dKw (y)  \quad \text{for \aev{} } y \in \Omega .
	\end{equation}
	Substituting \eqref{eq:riesz-maximal-bound} into \eqref{eq:tonelli-dual},
	and setting $f = \abs{\nabla v}$, we deduce
	\begin{equation}\label{eq:weighted-potential-bound}
		\int_{\Omega} (I_1 [\abs{\nabla v}] (z))^q w (z) \hspace{0.17em} \dd z
		\overset{{\eqref{eq:tonelli-dual},\eqref{eq:riesz-maximal-bound}}}{\leqslant} c_n dKa_d^{q - 1}  \int_{\Omega} \abs{\nabla v (y)}^q w (y)
		\hspace{0.17em} \dd y.
	\end{equation}
	Finally, applying the potential estimate \eqref{eq:sobolev-rep} from
	Lemma~\ref{lem:potential-representation}, we conclude
	\begin{equation}\label{eq:weighted-poincare-q}
\begin{aligned}
\int_{\Omega}\abs{v(z)-\mean{v}}^q w(z)\,\dd z
&\revadd{\leqslant}
A_\Omega^q\int_\Omega (I_1[\abs{\nabla v}](z))^q w(z)\,\dd z\\
&\revadd{\leqslant}
A_\Omega^q c_n dK a_d^{q-1}\int_\Omega\abs{\nabla v(y)}^q w(y)\,\dd y.
\end{aligned}
\end{equation}
	To ensure the final constant is uniform across the entire interval $q \in
	[1, P]$, we eliminate the explicit dependence on $q$ in the factors
	$A_{\Omega}^q$ and $a_d^{q - 1}$ by using $A_{\Omega}^q \leqslant
	\max \{A_{\Omega}, A_{\Omega}^P \}$ and $a_d^{q - 1} \leqslant \max \{a_d^0,
	a_d^{P - 1} \} = \max \{1, a_d^{P - 1} \}$. Substituting these endpoint
	bounds confirms that \eqref{eq:unweighted-mean} holds with the constant
	\begin{equation}\label{eq:weighted-poincare-constant}
		C_P = c_n dK \max \{A_{\Omega}, A_{\Omega}^P \} \max \{1, a_d^{P - 1} \},
	\end{equation}
	which depends on $P$ but is strictly independent of the specific choice of
	$q \in [1, P]$.
\end{proof}

\subsection{A common extension for translated power weights}
{ While general
weighted Sobolev extension theorems on Lipschitz domains are established in
{\cite{Chua1992}}, such results are typically formulated for a fixed exponent
$q$. Applying this general machinery directly would require tracking the
dependence of the extension constants on $q$ to verify their uniform
boundedness across $q \in [1,P]$.

To avoid this constant-tracking in general $A_p$ extension results, we exploit the specific geometric structure of our translated power
weights:
\begin{equation}
	w_a (z) = \abs{z - a}^{- \alpha}, \quad \text{with pole } a \in \Omega
	\text{ and exponent } 0 \leqslant \alpha < n.
\end{equation}
For $\alpha=0$, we set $w_a\equiv1$. For this family, reflection across
a fixed finite Lipschitz atlas provides a direct proof with fully
transparent and uniform control of the constants across all $q \in [1, P]$.
The restriction $a \in \Omega$ is essential to the geometric comparison below;
we do not assert that such a simple reflection estimate holds for arbitrary
$A_1$ weights.}

\begin{lemma}[Weighted reflection extension]
	\label{lem:weighted-tools}Let $\Omega \subset \R^n$ be a bounded
	Lipschitz domain. There exist a bounded linear operator $E : W^{1, 1}
	(\Omega) \to W^{1, 1} (\R^n)$ and a compact set $K \subset
	\R^n$, depending only on a fixed Lipschitz atlas and subordinate
	cutoffs, such that $Ev = v$ \aev{} in $\Omega$ and $\mathrm{supp} (Ev) \subset
	K$.
	
	For every $P \in [1, \infty)$, there exists $C_E = C_E  (n, \Omega, P) > 0$
	such that
	\begin{equation}
		\label{eq:weighted-extension} \int_{\R^n} \frac{\abs{Ev (z)}^q +
			\abs{\nabla Ev (z)}^q}{\abs{z - a}^{\alpha}} \hspace{0.17em} \dd z
		\leqslant C_E  \int_{\Omega} \frac{\abs{v (z)}^q + \abs{\nabla v
				(z)}^q}{\abs{z - a}^{\alpha}} \hspace{0.17em} \dd z
	\end{equation}
	for all $v \in W^{1, 1} (\Omega)$, $a \in \Omega$, $\alpha \in [0, n)$, and
	$q \in [1, P]$. The operator $E$ and the set $K$ are strictly independent of
	$P, q, a, \alpha$; the integrals are permitted to evaluate to $+ \infty$.
\end{lemma}

\begin{proof}
	We first construct the extension via local reflections, then establish the
	weighted estimate.\smallskip
	
	{ {\noindent}\emph{Step 1: Construction.} Choose a finite boundary atlas
	of graph cylinders. Geometrically, each cylinder $C_j$ is the region
	between two vertical translates of a boundary graph over an open cube.
	The boundary itself is described
	by a Lipschitz graph $z_n = \varphi_j (z')$, whose complement in the
	cylinder consists of two disjoint regions: an upper half $C_j^+$ and a lower half $C_j^-$.
	Because $\Omega$ is a Lipschitz domain, we can always orient our local rigid
	coordinates so that the interior of $\Omega$ lies strictly above this graph.
	Consequently, the upper half exactly captures the interior portion of the
	domain ($C_j^+ = C_j \cap \Omega$), while the lower half lies completely
	outside.
	
	Written in these local coordinates, the sets are defined as:
	\begin{equation}
		\begin{array}{ccl}
			C_j & = & \{(z', z_n) : z' \in Q'_j, \hspace{0.17em} \abs{z_n -
				\varphi_j (z')} < h_j \},\\
			C_j^{\pm} & = & \{z \in C_j : \pm (z_n - \varphi_j (z')) > 0\},\\
			C_j^+ & = & C_j \cap \Omega,
		\end{array} \label{eq:extension-graph-cylinder}
	\end{equation}
	where $Q'_j$ is an open cube, $h_j > 0$ is the vertical half-height, and
	$\varphi_j$ is Lipschitz with constant $M_j$.
	
	Choose an open interior set $U_0\Subset\Omega$ such that
	$\bar{\Omega}\subset U_0\cup\bigcup_{j=1}^m C_j$. By
	\cite[Lemma~9.3, p.~276]{Brezis2011}, choose smooth cutoffs
	$0\leqslant\psi_j\leqslant1$ subordinate to this cover, satisfying:
	\begin{equation}
		\begin{array}{rcl}
			\sum_{j = 0}^m \psi_j & = & 1 \quad \text{near } \bar{\Omega},\\
			K_0 & := & \mathrm{supp} \psi_0 \Subset U_0,\\
			K_j & := & \mathrm{supp} \psi_j \Subset C_j  \quad (1 \leqslant j
			\leqslant m) .
		\end{array} \label{eq:extension-partition}
	\end{equation}
	Discard identically zero boundary cutoffs and relabel the remaining
	patches. For $j \geqslant 1$, define the local reflection maps
	\begin{equation}
		\label{eq:extension-graph-reflection} R_j  (z', z_n) = (z', 2 \varphi_j
		(z') - z_n) .
	\end{equation}
	Each $R_j:C_j\to C_j$ is a bi-Lipschitz involution satisfying
	$R_j(C_j^-)=C_j^+$. Its derivative in these coordinates is
	\begin{equation}
		DR_j(z)\overset{\eqref{eq:extension-graph-reflection}}{=}\left(\begin{array}{cc}
			I_{n-1}&0\\
			2\nabla\varphi_j(z')^{\mathsf{T}}&-1
		\end{array}\right),\qquad
		\abs{\det DR_j}=1,\qquad \norm{DR_j}\leqslant1+2M_j
		\quad\text{\aev{}}
		\label{eq:reflection-jacobian}
	\end{equation}
	The change-of-variables formula for Lipschitz maps
	\cite[Section~3.3.3]{EvansGariepy2015} therefore shows that $R_j$
	preserves Lebesgue measure.\smallskip
}
	
	{\noindent}\emph{Step 2: Definition of the linear extension operator
		$E$.} We define the global extension piece by piece. For the interior region
	($j = 0$), the cutoff $\psi_0$ is compactly supported strictly inside
	$\Omega$. Therefore, the product $\psi_0 v$ vanishes identically in a
	neighborhood of the boundary $\partial \Omega$. We define $e_0 v$ by simply
	extending $\psi_0 v$ by zero to all of $\R^n$; the usual
	zero-extension property gives $e_0v\in W^{1,1}(\R^n)$
	\cite[Remark~4(ii), p.~265]{Brezis2011}.
	
	For the boundary regions ($j \geqslant 1$), we define the localized pieces
	$e_j v$ by preserving the function inside $\Omega$ and reflecting it across
	the boundary graph into the exterior:
	\begin{equation}\label{eq:local-extension}
		e_j v (z) = \left\{\begin{array}{ll}
			\psi_j (z) v (z), & \text{for } z \in C_j^+,\\
			\psi_j (z) v (R_j z), & \text{for } z \in C_j^-,\\
			0, & \text{for } z \in \R^n \setminus C_j .
		\end{array}\right.
	\end{equation}
	Values on the boundary graph, a set of Lebesgue measure zero, may be
	assigned arbitrarily. To justify the gluing across this graph, flatten
	$C_j$ by the bi-Lipschitz map
	$F_j(\xi',t)=(\xi',\varphi_j(\xi')+t)$ from
	$Q'_j\times(-h_j,h_j)$. The even reflection of
	$(v|_{C_j^+})\circ F_j$ belongs to
	$W^{1,1}(Q'_j\times(-h_j,h_j))$ by
	\cite[Lemma~9.2, pp.~273--275]{Brezis2011}, after rescaling the box.
	Composition with $F_j^{-1}$
	preserves $W^{1,1}$ by \cite[Theorem~2.2.2]{Ziemer1989}.
	Multiplication by $\psi_j$ and zero extension outside $C_j$ then give
	$e_jv\in W^{1,1}(\R^n)$, since $K_j\Subset C_j$.
	
	Summing these local pieces, we define the global extension operator:
	\begin{equation}
		\label{eq:extension-operator-definition} Ev = \sum_{j = 0}^m e_j v, \qquad
		\text{supported in the compact set} \quad K := \bigcup_{j = 0}^m K_j
		.
	\end{equation}
	Clearly, $E$ is a linear operator. Because $\sum_{j = 0}^m \psi_j = 1$
	inside $\Omega$, we recover $Ev=v$ and hence $\nabla Ev=\nabla v$
	almost everywhere in $\Omega$.
	
	We show that we can bound the magnitude and the weak gradient of the
	reflected pieces on the exterior region $C_j^-$. Precisely, we now show that
	the extension is pointwise controlled by the function's behavior in the
	interior:
	\begin{equation}
		\label{eq:reflection-gradient} \abs{(e_j v) (z)} + \abs{\nabla (e_j v)
			(z)} \leqslant b \left( \abs{v (R_j z)} + \abs{\nabla v (R_j z)} \right) 
		\quad \text{for \aev{} } z \in C_j^-,
	\end{equation}
	with
	\begin{equation}
		b := \max_{1 \leqslant j \leqslant m} \{1 + \norm{\nabla
			\psi_j}_{\infty}, 1 + 2 M_j \} \geqslant 1 \label{eq:geomconstb}
	\end{equation}
	a purely geometric constant. Indeed, by definition, $(e_j v) (z) = \psi_j
	(z) v (R_j z)$ for $z \in C_j^-$. Because $0 \leqslant \psi_j \leqslant 1$,
	the function itself is trivially bounded by $\abs{(e_j v) (z)} \leqslant
	\abs{v (R_j z)}$. To bound the gradient, we apply the standard Sobolev
	product and chain rules
	\cite[Remark~4(ii), p.~265]{Brezis2011} and
	\cite[Theorem~2.2.2]{Ziemer1989}, which yield:
	\begin{equation}\label{eq:local-extension-gradient}
		\nabla (e_j v) (z) \overset{\text{\cite{Brezis2011,Ziemer1989}}}{=}\nabla \psi_j (z) v (R_j z) + \psi_j (z)  (DR_j
		(z))^{\mathsf{T}} \nabla v (R_j z) .
	\end{equation}
	Taking the absolute value, and using the established bounds $\abs{\nabla
		\psi_j} \leqslant \norm{\nabla \psi_j}_{\infty}$ and $\norm{D  R_j}
	\leqslant 1 + 2 M_j$, we obtain the pointwise estimate:
	\begin{equation}\label{eq:local-extension-gradient-bound}
		\abs{\nabla (e_j v) (z)} \overset{{\eqref{eq:local-extension-gradient},\eqref{eq:reflection-jacobian}}}{\leqslant} \norm{\nabla \psi_j}_{\infty} \abs{v
			(R_j z)} + (1 + 2 M_j) \abs{\nabla v (R_j z)} .
	\end{equation}
	Adding the bounds for the function and its gradient together, we can
	uniformly dominate the coefficients on the right-hand side by defining the
	global constant $b$ in \eqref{eq:geomconstb}. This
	immediately guarantees \eqref{eq:reflection-gradient}.\smallskip
	
	{
	{\noindent}\emph{Step 3: The weighted estimate.} The aim of this step is to
	show that the singular weight $\abs{z - a}^{- \alpha}$ is quasi-preserved
	under reflection, i.e., that there exists a uniform structural constant $L
	\geqslant 1$, strictly independent of the chosen point $z$, the pole $a$,
	and the local patch $j$, such that for any $\alpha \in [0, n)$ the following
	weight bound holds:
	\begin{equation}
		\label{eq:reflection-weight} \frac{1}{\abs{z - a}^{\alpha}} \leqslant
		\frac{L^n}{\abs{R_j z - a}^{\alpha}}  \quad \text{for \aev{} } z \in K_j
		\cap C_j^- .
	\end{equation}
	For that, we must mathematically quantify and bound the distortion of the
	distance to the pole $a$ when a point $z \in K_j \cap C_j^-$ is reflected to
	$R_j z$. Specifically, we seek a uniform structural constant $L \geqslant 1$
	that bounds the multiplicative distortion ratio:
	\begin{equation}
		\label{eq:distortion-goal} \frac{\abs{R_j z - a}}{\abs{z - a}} \leqslant L.
	\end{equation}
	To establish this bound, note that the cutoff support $K_j$ is a compact set
	strictly contained inside the open cylinder $C_j$ and, therefore, $\delta_j
	:= \mathrm{dist} (K_j, \R^n \setminus C_j) > 0$. In particular,
	if a pole $a$ is located completely outside the cylinder, no point $z \in
	K_j$ can ever be closer to $a$ than this distance $\delta_j$.
	
	We now bound the vertical depth of the point $z$ below the boundary graph,
	which is precisely the quantity $\varphi_j (z') - z_n$. We consider two
	geometric cases for the location of the pole $a \in \Omega$: {\smallskip}
	
	{\noindent}{\emph{Case 1: The pole is inside the cylinder ($a \in \Omega
			\cap C_j$).}}
	Since $a$ is inside the domain, it sits strictly above the boundary graph,
	meaning $a_n > \varphi_j (a')$. We can connect $z$ to $a$ by adding and
	subtracting $\varphi_j (a')$ and $a_n$ to the depth:
	\begin{eqnarray}
		\label{eq:extension-vertical-distance-case1} \varphi_j (z') - z_n & = &
		(\varphi_j (z') - \varphi_j (a')) + (\varphi_j (a') - z_n) \nonumber\\
		& < & M_j \abs{z' - a'} + (a_n - z_n) \nonumber\\
		& \leqslant & M_j \abs{z - a} + \abs{z - a} = (M_j + 1) \abs{z - a} . 
	\end{eqnarray}
	Here, we used the Lipschitz continuity of the graph for the first term, and
	the geometric condition $\varphi_j (a') < a_n$ for the second.{\smallskip}
	
	{\noindent}{\emph{Case 2: The pole is outside the cylinder ($a \notin
			C_j$).}}
	By the definition of the cylinder $C_j$, the vertical depth of any point
	inside it is strictly bounded by the half-height $h_j$. Because $z \in K_j$
	and $a \notin C_j$, the buffer separation ensures $\abs{z - a} \geqslant
	\delta_j$. Multiplying $h_j$ by $1 \leqslant \abs{z - a} / \delta_j$ yields:
	\begin{equation}
		\label{eq:extension-vertical-distance-case2} \varphi_j (z') - z_n \overset{\eqref{eq:extension-graph-cylinder}}{<}h_j
		\overset{\abs{z-a}\geqslant\delta_j}{\leqslant} \frac{h_j}{\delta_j} \abs{z - a} .
	\end{equation}
	In either case, the vertical depth of $z$ below the boundary graph is
	linearly controlled by its distance to the pole $a$. Because the reflection
	$R_j$ acts strictly along the vertical fibers, the total displacement vector
	$R_j z - z$ is purely vertical, with a magnitude exactly twice this depth:
	$\abs{R_j z - z} = 2 (\varphi_j (z') - z_n)$.
	
	Applying the triangle inequality and substituting the bounds from the two
	cases yields:
	\begin{equation}\label{eq:reflection-distance}
		\abs{R_j z - a} \leqslant \abs{R_j z - z} + \abs{z - a} \overset{{\eqref{eq:extension-vertical-distance-case1},\eqref{eq:extension-vertical-distance-case2}}}{\leqslant} L \abs{z
			- a},
	\end{equation}
	where $L := \max_{1 \leqslant j \leqslant m} \{2 M_j + 3, 1 + 2 h_j /
	\delta_j \} \geqslant 1$. This $L$ is precisely the uniform structural
	constant required by \eqref{eq:distortion-goal}. Because it is constructed
	exclusively from the fixed atlas and cutoffs, $L$ is
	strictly independent of the evaluation point $z$, the pole $a$, and the
	local patch index $j$.
	
	As a direct consequence we get \eqref{eq:reflection-weight}: Since the
	exponent satisfies $\alpha \in [0, n)$ and the structural constant satisfies
	$L \geqslant 1$, taking the inverse $\alpha$-power of the distance bound
	yields:
	\begin{equation}
		\frac{1}{\abs{z - a}^{\alpha}} \overset{\eqref{eq:reflection-distance}}{\leqslant} \frac{L^{\alpha}}{\abs{R_j z -
				a}^{\alpha}} \overset{0\leqslant\alpha<n,\ L\geqslant1}{\leqslant} \frac{L^n}{\abs{R_j z - a}^{\alpha}}  \quad
		\text{for \aev{} } z \in K_j \cap C_j^- .
	\end{equation}
}
	{\noindent}\emph{Step 4: The global weighted bound.} Our final goal is
	to assemble the localized estimates to prove the global bound
	\eqref{eq:weighted-extension}. To shorten notation, let $G_q (z) :=
	\abs{v (z)}^q + \abs{\nabla v (z)}^q$ and define the interior weighted
	integral
	
	\begin{equation}\label{eq:extension-interior-energy}
		I := \int_{\Omega} \frac{G_q (z)}{\abs{z - a}^{\alpha}}
		\hspace{0.17em} \dd z.
	\end{equation}
	In terms of this shorthand, our overall objective is to demonstrate that
	\begin{equation}
		\int_{\R^n} \frac{\abs{(E  v) (z)}^q + \abs{\nabla (E  v)
				(z)}^q}{\abs{z - a}^{\alpha}} \hspace{0.17em} \dd z \leqslant C_E I.
	\end{equation}
	If $I=+\infty$, the estimate is immediate. We may therefore assume
	$I<\infty$. The proof reduces to
	decomposing the global integral into its interior and exterior components,
	estimating each boundary piece individually in terms of $I$, and finally
	synthesizing these pieces.{\smallskip}
	
	{\noindent}{\emph{Decomposition.}} Because $\Omega$ is Lipschitz, the
	boundary $\partial\Omega$ has Lebesgue measure zero. We split the
	integral into its contributions over $\Omega$ and
	$\R^n\setminus\Omega$. Since $Ev=v$ and
	$\nabla Ev=\nabla v$ \aev{} in $\Omega$, the interior contribution is $I$.
	Outside $\Omega$, the interior piece vanishes ($e_0 v =
	0$), meaning the global extension reduces purely to the sum of the $m$
	boundary pieces: $Ev = \sum_{j = 1}^m e_j v$. Applying the algebraic
	convexity bound $(\sum_{j = 1}^m t_j)^q \leqslant m^{q - 1}  \sum_{j = 1}^m
	t_j^q \leqslant m^{P - 1}  \sum_{j = 1}^m t_j^q$ to both the function and
	its gradient, we decompose the exterior integral into a sum over the local
	exterior regions $C_j^-$:
	\begin{equation}
		\label{eq:exterior-decomposition} \int_{\R^n \setminus \Omega}
		\frac{\abs{(Ev) (z)}^q + \abs{\nabla (Ev) (z)}^q}{\abs{z - a}^{\alpha}}
		\hspace{0.17em} \dd z \overset{\eqref{eq:extension-operator-definition}}{\leqslant} m^{P - 1}  \sum_{j = 1}^m \int_{C_j^-}
		\frac{\abs{(e_j v) (z)}^q + \abs{\nabla (e_j v) (z)}^q}{\abs{z -
				a}^{\alpha}} \hspace{0.17em} \dd z.
	\end{equation}
	{\emph{Local Estimation.}} Next, we estimate a single local piece from this
	sum. Using the algebraic convexity bound $(s + t)^q \leqslant 2^{q - 1} 
	(s^q + t^q)$ along with the pointwise gradient control
	\eqref{eq:reflection-gradient} and the weight bound
	\eqref{eq:reflection-weight}, we change variables via the measure-preserving
	reflection $y=R_jz$ (with $\dd y=\dd z$). Since $e_jv$ and its weak
	gradient vanish \aev{} outside $K_j$, we obtain:
	\begin{eqnarray}
		\int_{C_j^-} \frac{\abs{(e_j v) (z)}^q + \abs{\nabla (e_j v)
				(z)}^q}{\abs{z - a}^{\alpha}} \hspace{0.17em} \dd z &
		\overset{{\eqref{eq:reflection-gradient},\eqref{eq:reflection-weight}}}{\leqslant} & 2^{q - 1} b^q L^n 
		\int_{K_j \cap C_j^-} \frac{G_q  (R_j z)}{\abs{R_j z - a}^{\alpha}}
		\hspace{0.17em} \dd z \nonumber\\
		& \overset{\substack{\eqref{eq:reflection-jacobian}\\y=R_jz,\ q\leqslant P}}{\leqslant} & 2^{P - 1} b^P L^n 
		\int_{C_j^+} \frac{G_q (y)}{\abs{y - a}^{\alpha}} \hspace{0.17em} \dd y
		\nonumber\\
		& \overset{\eqref{eq:extension-interior-energy}}{\leqslant} & 2^{P - 1} b^P L^n I.  \label{eq:extension-local-weighted}
	\end{eqnarray}
	In the second step of \eqref{eq:extension-local-weighted}, we enlarged the
	integration domain from the reflected image $R_j  (K_j \cap C_j^-)$ to the
	full upper cylinder $C_j^+ \subset \Omega$. Crucially, because $b \geqslant
	1$ and $q \in [1, P]$, we uniformly bounded the $q$-dependent constants by
	their suprema at $q = P$.{\smallskip}
	
	{\noindent}{\emph{Synthesis.}} We now return to the original integral.
	Substituting the individual local estimates
	\eqref{eq:extension-local-weighted} back into the exterior decomposition
	\eqref{eq:exterior-decomposition} yields:
	\begin{equation}\label{eq:extension-exterior-bound}
		\int_{\R^n \setminus \Omega} \frac{\abs{(Ev) (z)}^q + \abs{\nabla
				(Ev) (z)}^q}{\abs{z - a}^{\alpha}} \hspace{0.17em} \dd z
		\overset{\eqref{eq:exterior-decomposition},
			\eqref{eq:extension-local-weighted}}{\leqslant} m^{P - 1}  \sum_{j = 1}^m
		(2^{P - 1} b^P L^n I) = 2^{P - 1} m^P b^P L^n I.
	\end{equation}
	Adding the interior contribution ($I$) to this exterior bound finally
	synthesizes the full global estimate:
	\begin{equation}
		\label{eq:extension-uniform-constant} \int_{\R^n} \frac{\abs{Ev
				(z)}^q + \abs{\nabla Ev (z)}^q}{\abs{z - a}^{\alpha}} \hspace{0.17em} \dd
		z \overset{\eqref{eq:extension-exterior-bound}}{\leqslant} \underbrace{(1 + 2^{P - 1} m^P b^P L^n)}_{C_E} I.
	\end{equation}
	Because the constant $C_E$ depends only on $P$ and the fixed geometric
	parameters $m, b$, and $L$, it is strictly independent of $q \in [1, P]$,
	the pole $a \in \Omega$, and the decay power $\alpha \in [0, n)$. Setting $q
	= P = 1$ and $\alpha = 0$ additionally confirms that $E$ maps boundedly into
	$W^{1, 1} (\R^n)$.
\end{proof}

Homogeneous extension results are available under suitable geometric
hypotheses; see, for example, \cite{Chua1992}. The specific localized
operator $E$ constructed in Lemma~\ref{lem:weighted-tools}, however, has
compact support. Consequently, a pure seminorm bound
\begin{equation}
	\int_{\R^n}\frac{\abs{\nabla Ev(z)}^q}{\abs{z-a}^{\alpha}}
	\hspace{0.17em}\dd z
	\leqslant C\int_{\Omega}\frac{\abs{\nabla v(z)}^q}{\abs{z-a}^{\alpha}}
	\hspace{0.17em}\dd z
\end{equation}
cannot hold for arbitrary functions. Indeed, applying it to $v\equiv1$
would give $\nabla E1=0$ \aev{} in $\R^n$. Thus $E1$ would be
constant, contradicting both $E1=1$ on $\Omega$ and its compact support.

The following corollary demonstrates that we can bypass this limitation. By
restricting our attention to functions with zero unweighted mean
($\mean{v} = 0$), we can use the uniform weighted Poincar{\'e}
inequality (Lemma~\ref{lem:unweighted-mean}) to completely absorb the $L^q$
norm into the gradient seminorm. This yields the desired
gradient-only extension estimate without requiring an entirely new homogeneous
extension construction.

\begin{corollary}[Uniform gradient control]
	\label{cor:gradient-extension}Let $E$ be the extension operator from
	Lemma~\ref{lem:weighted-tools}. For every $\alpha \in [0, n)$ and $P \in [1,
	\infty)$, there exists $C = C (n, \Omega, P, \alpha) > 0$ such that
	\begin{equation}
		\label{eq:uniform-extension} \int_{\R^n} \frac{\abs{\nabla (Ev)
				(z)}^q}{\abs{z - a}^{\alpha}} \hspace{0.17em} \dd z \leqslant C
		\int_{\Omega} \frac{\abs{\nabla v (z)}^q}{\abs{z - a}^{\alpha}}
		\hspace{0.17em} \dd z
	\end{equation}
	for all $a \in \Omega$, $q \in [1, P]$, and $v \in W^{1, 1} (\Omega)$
	satisfying the zero unweighted mean condition $\mean{v} = 0$.
\end{corollary}

\begin{proof}
	We must prove that the gradient seminorm of the extension is controlled
	solely by the gradient seminorm of the original function. To shorten
	notation, we set
	\begin{equation}\label{eq:gradient-interior-energy}
		J := \int_{\Omega} \frac{\abs{\nabla v (z)}^q}{\abs{z - a}^{\alpha}}
		\hspace{0.17em} \dd z.
	\end{equation}
	We may assume $J < \infty$, as the inequality is trivially true otherwise.
	
	Because $\alpha \in [0, n)$, the weight $\abs{z - a}^{- \alpha}$ is an $A_1$
	weight whose characteristic is bounded by a translation-invariant constant
	$K(n,\alpha)$ by Lemma~\ref{lem:power-A1}; see also
	\cite[Example~7.1.7, p.~506]{GrafakosClassical2014}. Thus, the uniform
	weighted Poincar{\'e} inequality
	(Lemma~\ref{lem:unweighted-mean}) applies with a constant $C_P = C_P  (n,
	\Omega, P, \alpha)$ that is strictly independent of the pole $a$ and the
	exponent $q \in [1, P]$.
	
	Since $v \in W^{1, 1} (\Omega)$ and $\mean{v} = 0$,
	Lemma~\ref{lem:unweighted-mean} immediately guarantees that $v$ belongs to
	the weighted space and satisfies:
	\begin{equation}
		\label{eq:extension-poincare-v} \int_{\Omega} \frac{\abs{v (z)}^q}{\abs{z
				- a}^{\alpha}} \hspace{0.17em} \dd z \overset{\mean{v}=0}{=}\int_{\Omega} \frac{\abs{v (z) -
				\mean{v}}^q}{\abs{z - a}^{\alpha}} \hspace{0.17em} \dd z
		\overset{\eqref{eq:unweighted-mean}}{\leqslant} C_P  \int_{\Omega} \frac{\abs{\nabla v (z)}^q}{\abs{z -
				a}^{\alpha}} \hspace{0.17em} \dd z = C_P J.
	\end{equation}
	We substitute this bound directly into the full extension estimate
	\eqref{eq:weighted-extension}. Discarding the non-negative term $\abs{E 
		v}^q$ on the left-hand side, we obtain:
	\begin{eqnarray}
		\int_{\R^n} \frac{\abs{\nabla (E  v) (z)}^q}{\abs{z - a}^{\alpha}}
		\hspace{0.17em} \dd z & \leqslant & \int_{\R^n} \frac{\abs{(E  v)
				(z)}^q + \abs{\nabla (Ev) (z)}^q}{\abs{z - a}^{\alpha}}
		\hspace{0.17em} \dd z \nonumber\\
		& \overset{\eqref{eq:weighted-extension}}{\leqslant} & C_E  \int_{\Omega} \frac{\abs{v (z)}^q + \abs{\nabla v
				(z)}^q}{\abs{z - a}^{\alpha}} \hspace{0.17em} \dd z \nonumber\\
		& \overset{\eqref{eq:extension-poincare-v}}{\leqslant} & C_E  (C_P J + J) .
		\label{eq:gradient-extension-final}
	\end{eqnarray}
	Factoring out $J$ yields the desired estimate \eqref{eq:uniform-extension}
	with the uniform constant $C = C_E  (1 + C_P)$.
\end{proof}
	
\section{The main theorem}\label{sec:main}

\begin{theorem}
	\label{thm:main}Let $\Omega\subset\R^n$ be a bounded Lipschitz
	domain, and let $B\subset\R^n$ be an open ball of diameter
	$D:=\diam(B)$ such that $\bar\Omega\subset B$. Let
	$p:\Omega\to[1,\infty)$ be measurable, with
	$1\leqslant p^-\leqslant p(x)\leqslant p^+<\infty$ \aev{}
	Let $\theta\in[0,1)$ and suppose that the measurable kernel
	$\kernel:\R^n\times\R^n\to[0,\infty)$ admits a
	measurable radial majorant $\kappa:\Omega\times(0,\infty)\to[0,\infty)$
	in the sense of \eqref{eq:radial-bound}.
	Define the radial coefficient
	\begin{equation}
		\label{eq:main-capacity}
		c(x):=\int_0^D r^{n+p(x)-2+\theta}\kappa(x,r)
		\hspace{0.17em}\dd r,\qquad x\in\Omega.
	\end{equation}
	There exists a constant $C=C(n,p^+,\theta,\Omega)>0$ such that, for every
	$u\in W^{1,1}(\Omega)$, in particular every
	$u\in W^{1,p(\cdot)}(\Omega)$,
	\begin{equation}
		\label{eq:main}
		\iint_{\Omega\times\Omega}\kernel(x,y)\abs{u(x)-u(y)}^{p(x)}
		\hspace{0.17em}\dd y\hspace{0.17em}\dd x
		\leqslant C\int_\Omega c(x)\left(
		\int_\Omega\frac{\abs{\nabla u(z)}^{p(x)}}{\abs{z-x}^{n-1+\theta}}
		\hspace{0.17em}\dd z\right)\dd x.
	\end{equation}
	Both sides are understood in $[0,\infty]$, with $0\cdot\infty=0$, as in
	Section~\ref{sec:prelim}. The constant $C$ is independent of $B$, $p$
	except through $p^+$, and the kernel.
\end{theorem}

\begin{proof}
	We subtract the mean, extend once, and apply Lemma~\ref{lem:morrey}
	on the containing ball. Set $\alpha:=n-1+\theta\in[0,n)$,
	$P:=p^+$, $v:=u-\mean{u}$, and $\tilde v:=Ev$, where $E$
	is the operator of Lemma~\ref{lem:weighted-tools}. Then
	$\tilde v\in W^{1,1}(\R^n)$, $\tilde v=v$ \aev{} on $\Omega$,
	$\mean{v}=0$, and $\nabla v=\nabla u$.
	
	If $u$ is constant \aev{}, both sides of \eqref{eq:main} vanish.
	Otherwise, connectedness of $\Omega$ implies that $\abs{\nabla u}>0$
	on a set of positive measure; see \cite[Remark~7, p.~269]{Brezis2011}.
	Consequently, the measurable function
	\begin{equation}
		\label{eq:main-inner-potential}
		I(x):=\int_\Omega
		\frac{\abs{\nabla u(z)}^{p(x)}}{\abs{z-x}^{\alpha}}
		\hspace{0.17em}\dd z
	\end{equation}
	is strictly positive. By Corollary~\ref{cor:gradient-extension},
	with $a=x$ and $q=p(x)\in[1,P]$,
	\begin{equation}
		\label{eq:main-extension-bound}
		\int_B\frac{\abs{\nabla\tilde v(z)}^{p(x)}}{\abs{z-x}^{\alpha}}
		\hspace{0.17em}\dd z
		\leqslant\int_{\R^n}
		\frac{\abs{\nabla\tilde v(z)}^{p(x)}}{\abs{z-x}^{\alpha}}
		\hspace{0.17em}\dd z
		\overset{\eqref{eq:uniform-extension}}{\leqslant}CI(x)
	\end{equation}
	for \aev{} $x\in\Omega$, with $C=C(n,P,\theta,\Omega)$.
	
	Apply Lemma~\ref{lem:morrey} directly to $\kernel$, with $U=B$,
	$S=\Omega$, and the fixed function $\tilde v$. The radial bound
	\eqref{eq:radial-bound} holds for \aev{} $y\in B$ and \aev{} $x\in\Omega$.
	Since $\tilde v=v$ \aev{} on $\Omega$, for \aev{} $x\in\Omega$ with
	$c(x)<\infty$ we obtain
	\begin{eqnarray}
		\label{eq:pointwise-final}
		\int_\Omega\kernel(x,y)\abs{u(x)-u(y)}^{p(x)}
		\hspace{0.17em}\dd y
		&\overset{\Omega\subset B}{\leqslant}&
		\int_B\kernel(x,y)\abs{\tilde v(x)-\tilde v(y)}^{p(x)}
		\hspace{0.17em}\dd y\nonumber\\
		&\overset{\eqref{eq:morrey-global}}{\leqslant}&
		c(x)\int_B\frac{\abs{\nabla\tilde v(z)}^{p(x)}}{\abs{z-x}^{\alpha}}
		\hspace{0.17em}\dd z\\
		&\overset{\eqref{eq:main-extension-bound}}{\leqslant}&
		Cc(x)I(x).\nonumber
	\end{eqnarray}
	Here Lemma~\ref{lem:morrey} already allows the measurable exponent
	$p(x)$: its proof uses the single full-measure set from
	Lemma~\ref{lem:ACL-radial} for $\tilde v$, independently of the exponent.
	Thus no intersection of exponent-dependent full-measure sets is needed.
	
	The chain \eqref{eq:pointwise-final} includes $c(x)=0$ under the
	convention $0\cdot\infty=0$. If $c(x)=\infty$, then
	$c(x)I(x)=\infty$, since $I(x)>0$, so its first and last terms still
	satisfy the required inequality. Integration in $x$ and Tonelli's
	theorem give \eqref{eq:main}. The constant is independent of $B$;
	its diameter enters only through \eqref{eq:main-capacity}.
\end{proof}

\begin{remark}
	\label{rem:diam-omega}For every $\varepsilon>0$, one may choose $B$
	with $D=2\diam(\Omega)+\varepsilon$: centre $B$ at any point of
	$\Omega$ and take radius $\diam(\Omega)+\varepsilon/2$.
	Then $\bar\Omega\subset B$, and $0<\varepsilon\leqslant\diam(\Omega)$
	makes $D$ comparable to $\diam(\Omega)$. For the fractional kernel
	of Corollary~\ref{cor:fractional}, this changes the constant but not
	the integrability condition at $r=0$. For the regularized kernel
	of Corollary~\ref{cor:regularised}, the finite cutoff makes the
	decay condition \eqref{eq:reg-condition} unnecessary.
\end{remark}

The computation of the radial coefficient in Corollary~\ref{cor:fractional}, now with
$D=\diam(B)$, gives the fractional Gagliardo-type bound on Lipschitz
domains.

\begin{corollary}
	\label{cor:variable-gagliardo}Let $\Omega\subset\R^n$ be a bounded
	Lipschitz domain, $p\in L^\infty(\Omega)$ with
	$1\leqslant p^-\leqslant p(x)\leqslant p^+<\infty$ \aev{}, and
	$\lambda\in(0,1)$. Suppose that $\theta\in[0,1)$ satisfies
	$\theta>1-(1-\lambda)p^-$. There exists a constant
	$C=C(n,\lambda,\theta,p^-,p^+,\Omega)>0$ such that, for every
	$u\in W^{1,1}(\Omega)$, in particular every
	$u\in W^{1,p(\cdot)}(\Omega)$,
	\begin{equation}
		\iint_{\Omega\times\Omega}
		\frac{\abs{u(x)-u(y)}^{p(x)}}{\abs{x-y}^{n+\lambda p(x)}}
		\dd y\hspace{0.17em}\dd x
		\leqslant C\iint_{\Omega\times\Omega}
		\frac{\abs{\nabla u(z)}^{p(x)}}{\abs{z-x}^{n-1+\theta}}
		\dd z\hspace{0.17em}\dd x.
	\end{equation}
\end{corollary}

\begin{proof}
	Fix a containing ball as in Remark~\ref{rem:diam-omega}, with diameter
	determined by $\Omega$. Take the kernel of Corollary~\ref{cor:fractional},
	with radial majorant $\kappa(x,r)=r^{-n-\lambda p(x)}$.
	Set $\beta(x):=(1-\lambda)p(x)-1+\theta$ and
	$\beta^\pm:=(1-\lambda)p^\pm-1+\theta$. Since
	$0<\beta^-\leqslant\beta(x)\leqslant\beta^+$ \aev{},
	\begin{equation}
		\label{eq:fractional-capacity}
		c(x)\overset{\eqref{eq:main-capacity}}{=}
		\int_0^D r^{\beta(x)-1}\hspace{0.17em}\dd r
		=\frac{D^{\beta(x)}}{\beta(x)}
		\leqslant\frac{\max\{D^{\beta^-},D^{\beta^+}\}}{\beta^-}.
	\end{equation}
	The conclusion follows from \eqref{eq:main} and
	\eqref{eq:fractional-capacity}.
\end{proof}

{\color{black}
\begin{remark}[Luxemburg-gauge consequence]\label{rem:luxemburg-consequence}
Under the hypotheses of Theorem~\ref{thm:main}, put
\begin{align*}
\mathcal A(u)&=\iint_{\Omega\times\Omega}\kernel(x,y)\abs{u(x)-u(y)}^{p(x)}\,\dd y\,\dd x,\\
\mathcal B(F)&=\int_\Omega c(x)\int_\Omega
\frac{\abs{F(z)}^{p(x)}}{\abs{z-x}^{n-1+\theta}}\,\dd z\,\dd x,
\end{align*}
and define the possibly infinite Luxemburg gauges
\[
[u]_{\mathcal A}=\inf\{t>0:\mathcal A(u/t)\leqslant1\},\qquad
\|F\|_{\mathcal B}=\inf\{t>0:\mathcal B(F/t)\leqslant1\},
\]
with $\inf\varnothing=\infty$. Then
\begin{equation}\label{eq:luxemburg-consequence}
[u]_{\mathcal A}\leqslant
\max\{C^{1/p^-},C^{1/p^+}\}\,\|\nabla u\|_{\mathcal B}.
\end{equation}
Indeed, set $a=\max\{C^{1/p^-},C^{1/p^+}\}$. For every $t>0$ with
$\mathcal B(\nabla u/t)\leqslant1$, apply the theorem to $u/(at)$ and use
$a^{-p(x)}\leqslant C^{-1}$ to obtain $\mathcal A(u/(at))\leqslant1$.
Taking the infimum over such $t$ proves the claim; if there is no such $t$,
the right-hand side is infinite. This also covers a zero right-hand gauge
by taking arbitrarily small admissible $t$.
For the fractional kernel, $[u]_{\mathcal A}$ is the associated
variable-exponent Gagliardo seminorm. Similarly, Corollary~\ref{cor:modularPW}
gives a bound for $\|u-\mean{u}\|_{p(\cdot)}$ by the gauge obtained from
$\mathcal B$ with $c=1$ and $\theta=0$. These bounds retain the mixed
gradient functional; they do not replace it by $\|\nabla u\|_{p(\cdot)}$.
\end{remark}
}

\begin{remark}
	\label{rem:finite-rhs}Membership in $W^{1,p(\cdot)}(\Omega)$ alone
	need not make the right-hand side finite; see
	Example~\ref{ex:infinite-modular}. A sufficient condition for the
	right-hand side in Corollary~\ref{cor:variable-gagliardo} to be finite
	is $u\in W^{1,p^+}(\Omega)$. Indeed,
	$\abs{\nabla u(z)}^{p(x)}\leqslant1+\abs{\nabla u(z)}^{p^+}$, and
	\begin{equation}\label{eq:lipschitz-kernel-mass}
		\sup_{z\in\Omega}\int_\Omega
		\abs{x-z}^{-(n-1+\theta)}\hspace{0.17em}\dd x
		\overset{\eqref{eq:kernel-mass-theta}}{\leqslant}\frac{\abs{\sphere}\hspace{0.17em}
			\diam(\Omega)^{1-\theta}}{1-\theta}.
	\end{equation}
	Tonelli's theorem gives the claimed finiteness. The same conclusion
	holds in Theorem~\ref{thm:main} if $c\in L^\infty(\Omega)$.
	No such finiteness assumption is required for the estimates themselves.
\end{remark}

At $\theta=0$, taking $\kernel\equiv1$ in Theorem~\ref{thm:main} and
applying Jensen's inequality in $y$ with respect to
$\dd y/\mathcal{L}^n(\Omega)$ gives the modular Poincar\'e--Wirtinger inequality
of \cite{DDFH2024}. We give a shorter proof from
Lemma~\ref{lem:potential-representation}.

\begin{corollary}[Modular Poincar\'e--Wirtinger inequality]
	\label{cor:modularPW}Let $\Omega\subset\R^n$ be a bounded Lipschitz
	domain and let $p\in L^\infty(\Omega)$ with $p\geqslant1$ \aev{}
	There exists a constant $C=C(n,p^+,\Omega)>0$ such that, for every
	$u\in W^{1,1}(\Omega)$, in particular every
	$u\in W^{1,p(\cdot)}(\Omega)$,
	\begin{equation}
		\label{eq:modularPW}
		\int_\Omega\abs{u(x)-\mean{u}}^{p(x)}\dd x
		\leqslant C\iint_{\Omega\times\Omega}
		\frac{\abs{\nabla u(z)}^{p(x)}}{\abs{z-x}^{n-1}}
		\hspace{0.17em}\dd z\hspace{0.17em}\dd x.
	\end{equation}
	For constant $p$, this implies the classical Poincar\'e--Wirtinger
	inequality \eqref{eq:classical-PW}.
\end{corollary}

\begin{proof}
	Put $d=\diam(\Omega)$, $a_d=\abs{\sphere}\hspace{0.17em}d$, and $P=p^+$.
	The potential estimate of Lemma~\ref{lem:potential-representation}
	and the kernel form of Jensen's inequality give, for \aev{} $x\in\Omega$,
	\begin{eqnarray}
		\abs{u(x)-\mean{u}}^{p(x)}
		&\overset{\eqref{eq:sobolev-rep}}{\leqslant}&
		A_\Omega^{p(x)}\left(\int_\Omega
		\frac{\abs{\nabla u(z)}}{\abs{z-x}^{n-1}}\hspace{0.17em}\dd z
		\right)^{p(x)}\nonumber\\
		&\overset{\eqref{eq:kernel-jensen}}{\leqslant}&
		A_\Omega^{p(x)}a_d^{p(x)-1}\int_\Omega
		\frac{\abs{\nabla u(z)}^{p(x)}}{\abs{x-z}^{n-1}}
		\hspace{0.17em}\dd z.
	\end{eqnarray}
	Both inequalities hold in the extended sense. Integrating in $x$
	proves \eqref{eq:modularPW} with
	\begin{equation}
		\label{eq:modularPW-constant}
		C=\max\{A_\Omega,A_\Omega^P\}\max\{1,a_d^{P-1}\}.
	\end{equation}
	For constant $p$, Tonelli's theorem and
	$\sup_{z\in\Omega}\int_\Omega\abs{x-z}^{1-n}\hspace{0.17em}\dd x
	\leqslant a_d$ give \eqref{eq:classical-PW}.
\end{proof}

\begin{remark}
	Corollary~\ref{cor:modularPW} proves the estimate of
	\cite[Theorem~1.1]{DDFH2024} for every $u\in W^{1,1}(\Omega)$,
	hence for the full variable-exponent Sobolev class. The argument
	uses the potential estimate in Lemma~\ref{lem:potential-representation};
	its constant depends on the fixed Lipschitz atlas and the classical
	$L^1$ Poincar\'e constant through $A_\Omega$.
	The proof in \cite{DDFH2024} instead gives an explicit constant
	in terms of a starshaped decomposition of $\Omega$.
\end{remark}

	{
	
	\section*{Declarations}

	\noindent
	\textit{Conflict of Interest:} The authors declare that they have no conflicts of interest.
	\smallskip

	\section*{Acknowledgements}
	\textsc{G.D.F.} is a member of GNAMPA--INdAM. He gratefully acknowledges partial financial support from the GNAMPA Project CUP\_E53C25002010001, and from the University of Naples Federico II through the FRA Project-B ``VarMoCry'' on  \emph{Variational Analysis and Modeling of Liquid Crystals}.
	
	The research of L.H. has been partially supported by the Austrian Science Fund (FWF) through project  \href{https://doi.org/10.55776/P34609}{10.55776/P34609} on \emph{the Analysis and Modeling of Magnetic Skyrmions}.

}

\section*{\texorpdfstring{\revadd{Declaration of generative AI and AI-assisted technologies}}{Declaration of generative AI and AI-assisted technologies}}
\revadd{During the preparation of this work, the authors used Gemini for language editing and for assistance with TeXmacs-to-LaTeX conversion and typesetting. Gemini was also used to examine mathematical arguments for possible gaps and assist with literature searches. The authors retain responsibility for independently verifying all arguments and references and for the accuracy and integrity of the final manuscript.}


\begin{thebibliography}{99}
	
	\bibitem{BahrouniRadulescu2018}
	A.~Bahrouni, V.~D.~R\u{a}dulescu,
	\emph{On a new fractional Sobolev space and applications
		to nonlocal variational problems with variable exponent},
	Discrete Contin.\ Dyn.\ Syst.\ Ser.\ S
	\textbf{11} (2018), no.~3, 379--389.
	
	\bibitem{BBM2001}
	J.~Bourgain, H.~Brezis, P.~Mironescu,
	\emph{Another look at Sobolev spaces},
	in: Optimal Control and Partial Differential Equations,
	J.~L.~Menaldi, E.~Rofman and A.~Sulem (eds.),
	IOS Press, Amsterdam, 2001, pp.~439--455.
	
	\bibitem{Brezis2011}
	H.~Brezis,
	\emph{Functional Analysis, Sobolev Spaces and
		Partial Differential Equations},
	Universitext, Springer, New York, 2011.
	
	\bibitem{Chua1992}
	S.-K.~Chua,
	\emph{Extension theorems on weighted Sobolev spaces},
	Indiana Univ.\ Math.\ J.\
	\textbf{41} (1992), no.~4, 1027--1076.
	
	\bibitem{CiarletDinca2011}
	P.~G.~Ciarlet, G.~Dinca,
	\emph{A Poincar\'e inequality in a Sobolev space
		with a variable exponent},
	Chinese Ann.\ Math.\ Ser.\ B
	\textbf{32} (2011), no.~3, 333--342.
	
	\bibitem{CruzUribeDiFrattaFiorenza2018}
	D.~V.~Cruz-Uribe, G.~Di Fratta, A.~Fiorenza,
	\emph{Modular inequalities for the maximal operator
		in variable Lebesgue spaces},
	Nonlinear Anal.\
	\textbf{177} (2018), part A, 299--311.
	
	\bibitem{CruzUribeFiorenza2013}
	D.~V.~Cruz-Uribe, A.~Fiorenza,
	\emph{Variable Lebesgue Spaces:
		Foundations and Harmonic Analysis},
	Applied and Numerical Harmonic Analysis,
	Birkh\"auser, Basel, 2013.
	
	\bibitem{CruzUribeLopezGarciaOjea2026}
\revadd{D.~Cruz-Uribe OFS, F.~L\'opez-Garc\'ia, I.~Ojea,
\emph{Poincar\'e and Sobolev inequalities with variable exponents and
log-H\"older continuity only at the boundary},
Potential Anal.\ 	extbf{64} (2026), Paper No.~38.
doi:10.1007/s11118-025-10249-6.}

	\bibitem{DDFH2024}
	E.~Davoli, G.~Di Fratta, A.~Fiorenza, L.~Happ,
	\emph{A modular Poincar\'e--Wirtinger inequality
		for Sobolev spaces with variable exponents},
	Nonlinear Differ.\ Equ.\ Appl.\
	\textbf{31} (2024), no.~5, Paper No.~81.
	Preprint version: arXiv:2304.13132.
	
	\bibitem{DavoliDiFrattaGiorgioPinamonti2025}
	E.~Davoli, G.~Di Fratta, R.~Giorgio, A.~Pinamonti,
	\emph{Necessary and sufficient conditions for the
		Maz'ya--Shaposhnikova formula in (fractional)
		Sobolev spaces},
	Preprint, arXiv:2509.23226 (2025).
	
	\bibitem{DavoliDiFrattaPagliari2026}
	E.~Davoli, G.~Di Fratta, V.~Pagliari,
	\emph{Sharp conditions for the validity of the
		Bourgain--Brezis--Mironescu formula},
	Proc.\ Roy.\ Soc.\ Edinburgh Sect.\ A
	\textbf{156} (2026), no.~1, 69--92.
	doi:10.1017/prm.2024.47.
	
	\bibitem{DHHR2011}
	L.~Diening, P.~Harjulehto, P.~H\"ast\"o,
	M.~R\r{u}\v{z}i\v{c}ka,
	\emph{Lebesgue and Sobolev Spaces with Variable Exponents},
	Lecture Notes in Mathematics \textbf{2017},
	Springer, Heidelberg, 2011.
	
	\bibitem{DiFrattaFiorenza2022}
	G.~Di Fratta, A.~Fiorenza,
	\emph{A unified divergent approach to
		Hardy--Poincar\'e inequalities in classical
		and variable Sobolev spaces},
	J.\ Funct.\ Anal.\
	\textbf{283} (2022), no.~5, Paper No.~109552.
	
	\bibitem{DiFrattaFiorenza2025}
	G.~Di Fratta, A.~Fiorenza,
	\emph{A generalized radial integration by parts formula
		and its applications to
		Caffarelli--Kohn--Nirenberg inequalities},
	Anal.\ Math.\ Phys.\
	\textbf{15} (2025), Paper No.~62.
	doi:10.1007/s13324-025-01060-y.
	
		\bibitem{Fiorenza2021}
	A.~Fiorenza,
	\emph{Modulars from Nakano onwards},
	Constructive \ Math.\ Anal.\
	\textbf{4} (2021), no.~2, 145--178.

	
	\bibitem{EvansGariepy2015}
	L.~C.~Evans, R.~F.~Gariepy,
	\emph{Measure Theory and Fine Properties of Functions},
	revised edition, Textbooks in Mathematics,
	CRC Press, Boca Raton, FL, 2015.
	
	\bibitem{FanZhangZhao2005}
	X.~Fan, Q.~Zhang, D.~Zhao,
	\emph{Eigenvalues of $p(x)$-Laplacian Dirichlet problem},
	J.\ Math.\ Anal.\ Appl.\
	\textbf{302} (2005), no.~2, 306--317.
	
	\bibitem{GilbargTrudinger2001}
	D.~Gilbarg, N.~S.~Trudinger,
	\emph{Elliptic Partial Differential Equations of Second Order},
	reprint of the 1998 edition, Classics in Mathematics,
	Springer, Berlin, 2001.
	
	\bibitem{GrafakosClassical2014}
	L.~Grafakos,
	\emph{Classical Fourier Analysis}, 3rd edition,
	Graduate Texts in Mathematics \textbf{249},
	Springer, New York, 2014.
	
	\bibitem{HurriSyrjanenVahakangas}
	R.~Hurri-Syrj\"anen, A.~V.~V\"ah\"akangas,
	\emph{On fractional Poincar\'e inequalities},
	J.\ Anal.\ Math.\
	\textbf{120} (2013), 85--104.
	
	\bibitem{KaufmannRossiVidal2017}
	U.~Kaufmann, J.~D.~Rossi, R.~Vidal,
	\emph{Fractional Sobolev spaces with variable exponents
		and fractional $p(x)$-Laplacians},
	Electron.\ J.\ Qual.\ Theory Differ.\ Equ.\
	(2017), no.~76, 1--10.
	
	\bibitem{Maeda2008}
	F.-Y.~Maeda,
	\emph{Poincar\'e type inequalities for variable exponents},
	JIPAM J.\ Inequal.\ Pure Appl.\ Math.\
	\textbf{9} (2008), no.~3, Article 68, 5 pp.
	
	\bibitem{Ponce2004}
	A.~C.~Ponce,
	\emph{An estimate in the spirit of Poincar\'e's inequality},
	J.\ Eur.\ Math.\ Soc.\ (JEMS)
	\textbf{6} (2004), no.~1, 1--15.
	
	\bibitem{Ziemer1989}
	W.~P.~Ziemer,
	\emph{Weakly Differentiable Functions:
		Sobolev Spaces and Functions of Bounded Variation},
	Graduate Texts in Mathematics \textbf{120},
	Springer, New York, 1989.
	
\end{thebibliography}
\end{document}